\documentclass[10pt,letterpaper]{amsart}
\pdfoutput=1
\usepackage[T1]{fontenc}
\usepackage[latin9]{inputenc}
\usepackage{geometry}
\usepackage{babel}
\usepackage{mathtools}
\usepackage{amsbsy}
\usepackage{amstext}
\usepackage{amsthm}
\usepackage{amssymb}
\usepackage{graphicx}
\usepackage{verbatim}
\usepackage[unicode=true,pdfusetitle, bookmarks=true,bookmarksnumbered=false,bookmarksopen=false, breaklinks=false,pdfborder={0 01},backref=false,colorlinks=true] {hyperref}
\usepackage{cleveref}
\usepackage[all]{xy}
\usepackage{xcolor}
\usepackage{soul}
\makeatletter

\numberwithin{equation}{section}
\numberwithin{figure}{section}
\theoremstyle{plain}
\newtheorem{thm}{\protect\theoremname}[section]
\theoremstyle{plain}
\newtheorem{cor}[thm]{\protect\corollaryname}
\theoremstyle{plain}
\newtheorem{lem}[thm]{Lemma}
\newtheorem{prop}[thm]{Proposition}
\theoremstyle{definition}
\newtheorem{defn}[thm]{\protect\definitionname}
\theoremstyle{definition}
\newtheorem{exmp}[thm]{Example}
\theoremstyle{definition}
\newtheorem{qn}{Question}
\newtheorem{conj}{Conjecture}
\theoremstyle{remark}
\newtheorem{rem}[thm]{\protect\remarkname}

\newcommand{\Arf}{\operatorname{Arf}}
\newcommand{\AZ}{\mathsf{AZ}}

\newcommand{\conjugate}{\operatorname{conj}}
\newcommand{\id}{\operatorname{id}}
\newcommand{\tr}{\operatorname{tr}}
\newcommand{\cotr}{\operatorname{cotr}}
\newcommand{\image}{\operatorname{Im}}
\newcommand{\Inv}{\operatorname{Inv}}
\newcommand{\mathhat}{\operatorname{Hat}}
\newcommand{\mathminus}{\operatorname{Minus}}
\newcommand{\pt}{\operatorname{pt}}
\newcommand{\spa}{\operatorname{span}}
\def\spinc{\textrm{Spin}^c}
\def\Sym{\mathrm{Sym}}

\author{Sungkyung Kang}
\address{Center for Geometry and Physics, Institute for Basic Science (IBS), Pohang 37673, Korea}
\email{sungkyung38@icloud.com}
\author{JungHwan Park}
\address{Department of Mathematical Sciences, Korea Advanced Institute for Science and Technology}
\email{jungpark0817@kaist.ac.kr}

\makeatother

\providecommand{\corollaryname}{Corollary}
\providecommand{\definitionname}{Definition}
\providecommand{\remarkname}{Remark}
\providecommand{\theoremname}{Theorem}

\begin{document}

\title{Torsion in the knot concordance group and cabling}

\begin{abstract}
We define a nontrivial mod 2 valued additive concordance invariant defined on the torsion subgroup of the knot concordance group using involutive knot Floer package. For knots not contained in its kernel, we prove that their iterated $(\text{odd},1)$-cables have infinite order in the concordance group and, among them, infinitely many are linearly independent. Furthermore, by taking $(2,1)$-cables of the aforementioned knots, we present an infinite family of knots which are strongly rationally slice but not slice.

\end{abstract}
\maketitle

\section{Introduction}
Given a knot $K$ and relatively prime integers $p$ and $q$, let $K_{p,q}$ be the \emph{$(p,q)$-cable} of $K$ where $p$ denotes the longitudinal winding. Through out the paper we assume that $|p|>1$ since $K_{1,q}$ is isotopic to $K$. Recall that a knot is called \emph{slice} if it bounds a smoothly embedded disk in $B^4$ and the fact that $(p,1)$-cable and $(p,-1)$-cable of a slice knot are both slice. Miyazaki~\cite[Question 3]{Miyazaki:1994-1} asked if the converse of this fact also holds (see also \cite[Conjecture 1.3]{Meier;2017-1} and \cite[Section 1]{Cochran-Davis-Ray:2014-1}). In fact, we go one step further and focus on the following stronger conjecture. Let $\mathcal{C}$ denote the \emph{smooth knot concordance group}.
\begin{conj}\label{conj:main}
The $(p,q)$-cable of $K$ is torsion in $\mathcal{C}$ if and only if $K$ is slice and $|q|=1$.
\end{conj}

It can be easily seen that if $|q| >1$, then $K_{p,q}$ has infinite order in $\mathcal{C}$.\footnote{One way to prove this is to use the cabling formulae for the Ozsv\'ath-Szab\'o concordance invariant $\tau$~\cite{Oz-Sz:2003-1} and the Ozsv\'{a}th-Stipsicz-Szab\'{o} concordance invariant $\Upsilon$~\cite{Oz-St-Sz:2017-1}; see e.g.\ \cite[Theorem 1]{Hom:2014-1} for $\tau$ and \cite[Proposition 5.3]{Feller-Park-Ray:2019-1} for $\Upsilon$.} Therefore, for the rest of the article we only consider the case when $K$ is not slice and $|q|=1$. In fact, since $K_{p,q}$ is isotopic to $K_{-p,-q}$, we will further assume that $q=1$.

There are many nonslice knots where their $(p,1)$-cables have infinite order in $\mathcal{C}$. For example, if $K$ has nontrivial signature function~\cite{Tristram:1969-1, Levine:1969-1}, then a formula of Litherland~\cite{Litherland:1977-1} implies that $K_{p,1}$ has infinite order in $\mathcal{C}$. In fact, the same formula implies that the set of cables $\{K_{p,1}\}_{p>1}$ is linearly independent in $\mathcal{C}$. A similar conclusions obtained by various different techniques can be found in~\cite{kim2018infinite,Feller-Park-Ray:2019-1,chen2021upsilon,davis2021linear}. These results crucially use the fact that the starting knot $K$ has infinite order in $\mathcal{C}$ (e.g.\ a knot with nontrivial signature function has infinite order in $\mathcal{C}$). 

On the other hand, determining the order of $(p,1)$-cables of a \emph{torsion} knot is much harder. The only known result of this kind is given recently in~\cite{hom2020linear}. They showed that when $K$ is the figure-eight knot, the set of cables $\{K_{2n+1,1}\}_{n>0}$ is linearly independent in $\mathcal{C}$ using involutive knot Floer homology~\cite{hendricks2017involutive}. (This also gave a first example of an infinite order rationally slice knot.) In this article, we expand this result to a much larger family. Roughly, we show that ``half'' of the torsion knots have infinite order once cabled (compare it with \cite[Theorem 8.5]{Miyazaki:1994-1} where they consider cables of fibered knots). Let $K_{p_1,q_1;p_2,q_2;\ldots ;p_m,q_m}$ denote the iterated cable of $K$.

\begin{thm}
\label{thm:mainthm}
Denote by $\mathcal{C}_{T}$ the torsion subgroup of the knot concordance group $\mathcal{C}$. Then there is a nontrivial group homomorphism 
$$\mathfrak{A}\colon \mathcal{C}_{T}\rightarrow \mathbb{Z}/2\mathbb{Z}.$$
Moreover, if $K$ is torsion in $\mathcal{C}$ with $\mathfrak{A}(K) = 1$, then for any sequence of positive integers $n_1, n_2, \ldots, n_m$ the iterated cable $K_{2n_1+1,1;2n_2+1,1;\ldots ;2n_m+1,1}$ has infinite order in $\mathcal{C}$. In particular, for any nonzero integer $n$ the cable $K_{2n+1,1}$ has infinite order in $\mathcal{C}$.
\end{thm}

We also show that knots obtained by certain cables on $K$ with $\mathfrak{A}(K) = 1$ are linearly independent. 

\begin{thm}
\label{thm:mainthm-indep}
If $K$ is torsion in $\mathcal{C}$ with $\mathfrak{A}(K) = 1$, then the set of cables $\{K_{2n+1,1}\}_{n>0}$ contains an infinite subset which is linearly independent in $\mathcal{C}$.
\end{thm}

For the nontrivial group homomorphism $\mathfrak{A}$, we first define the almost local $\iota_K$-equivalence group $\mathfrak{I}^U_K$ of horizontal almost $\iota_K$-complexes (see Section~\ref{sec:horizontal} for more details). Note that the $\iota_K$-local equivalence group $\mathfrak{I}_K$ and the almost $\iota_K$-local equivalence group $\widehat{\mathfrak{I}}_K$ were defined previously in \cite{zemke2019connected} and \cite{hom2020linear}, respectively. Moreover, we show that $\mathfrak{A}$ is the composition of the following sequence of the natural ``forgetful maps'' followed by the involutive knot Floer homology map $\mathcal{C}\rightarrow \mathfrak{I}_K$;
\[
\mathfrak{I}_K \rightarrow \widehat{\mathfrak{I}}_K \rightarrow \mathfrak{I}^U_K.
\]
In other words, involutive knot Floer theory gives a well-defined homomorphism $\mathfrak{A}\colon\mathcal{C}\rightarrow \mathfrak{I}^U_K$.


It turns out that $\mathfrak{I}^U_K$ is totally ordered modulo the figure-eight knot $E$, so that its torsion subgroup is generated by $E$. This allows us to consider its restriction $\mathfrak{A}\colon\mathcal{C}_{T}\rightarrow \mathbb{Z}/2\mathbb{Z}$, where we have that $\mathfrak{A}(E)=1$. Therefore, the above theorems can be thought of it as a generalizations of \cite{hom2020linear}. Moreover, note that the ``half'' of the torsion knots satisfy $\mathfrak{A}=1$; i.e.\ if $\mathfrak{A}(K)=0$, then $\mathfrak{A}(K\# E)=1$.

The phenomenon of total orderings of algebraic structures arising from Heegaard Floer theory and its variants occurs in other settings as well. For instance, the almost local equivalence group of $\iota$-complexes is totally ordered~\cite{dai2018infinite} and the almost local equivalence group of knot-like complexes over $\mathbb{F}_2[U,V]/(UV)$ is also totally ordered~\cite{DHST:2021-1}. A notable difference in our case is that our group $\mathfrak{I}^U_K$ contains a torsion, and one has to take a quotient by that element to obtain a total ordering.

Now, we restrict our attention to the case when $K$ allows a special symmetry and $p=2$. Note that, so far, we have only considered $(\text{odd},1)$-cables, but by considering $(2,1)$-cables we will obtain an interesting application to rational knot concordance. We say a knot $K$ is \emph{rationally slice} if it bounds a smoothly embedded disk $\Delta$ in a rational homology ball $X$. Moreover, if the inclusion induces an isomorphism
$$H_1(S^3 \smallsetminus K;\mathbb{Z}) \xrightarrow{\cong} H_1(X \smallsetminus \Delta; \mathbb{Z})/\text{torsion},$$
then we say $K$ is \emph{strongly rationally slice}.\footnote{Strongly rationally slice knots are rationally slice knots with \emph{complexity} one. See~\Cref{sec:stronglyrationallyslice} for the notion of complexity and \cite{Cochran-Orr:1993-1, Cha-Ko:2002-1, Cha:2007-1} for related works.} Following the proof in~\cite{Levine:1969-1}, it can be easily verified that every strongly rationally slice knot is algebraically slice. Moreover, slice obstructions such as $\tau$-invariant~\cite{Oz-Sz:2003-1}, $\varepsilon$-invariant~\cite{Hom:2014-2}, $\Upsilon$-invariant~\cite{Oz-St-Sz:2017-1}, $\Upsilon^2$-invariant~\cite{Kim-Livingston:2018-1}, $\nu^+$-invariant~\cite{Hom-Wu:2016-1}, and $\varphi_j$-invariants~\cite{DHST:2021-1} all vanish for strongly rationally slice knots. These facts make it difficult to find an example of a strongly rationally slice knot that is not slice. In this article, we present the first such examples:
\begin{thm}
\label{thm:mainthm-rationallyslice}
There are infinitely many pairwise nonconcordant knots which are strongly rationally slice but not slice.
\end{thm}

The examples are constructed in the following way. Recall that a knot $K$ is called \emph{strongly negative-amphichiral} if there is an orientation-preserving involution $\phi\colon S^3 \to S^3$ such that $\phi(K)=K$ and the fixed point set of $\phi$ is a copy of $S^0$ contained in $K$. Kawauchi~\cite{kawauchi2009rational} showed that every strongly negative-amphichiral knot $K$ is rationally slice; that is, $K$ bounds a smoothly embedded disk $\Delta$ in a rational homology ball $X$. Moreover, he showed the homomorphism induced by the inclusion sends the generator of $H_1(S^3 \smallsetminus K;\mathbb{Z})$ to twice the generator of $H_1(X \smallsetminus \Delta; \mathbb{Z})/\text{torsion}$. Combined with Mayer-Vietoris sequence (see \Cref{lem:satellitecomplexity}), it implies that the iterated cable $K_{p_1,1;p_2,1;\ldots;p_m,1}$ of a strongly negative-amphichiral knot $K$ is strongly rationally slice if $p_i$ is even for some $i$. Hence the following theorem immediately implies Theorem~\ref{thm:mainthm-rationallyslice}.

\begin{thm}
\label{thm:mainthm-complexity}
If $K$ is torsion in $\mathcal{C}$ with $\mathfrak{A}(K) = 1$, then for any sequence of positive integers $n_1, n_2, \ldots, n_m$ the $(2,1)$-cable of the iterated cable $K_{2n_1+1,1;2n_2+1,1; \ldots; 2n_m+1,1}$ has infinite order in $\mathcal{C}$.
\end{thm}

\subsection*{Organization and proof outline}
We briefly sketch the ideas for the proofs of above theorems as follows. In \Cref{sec:horizontal}, we define the concept of horizontal almost $\iota_K$-complexes, which provides an algebraic model for $CFK^-$ of knots in $S^3$, together with the action of $\iota_K$ on $\widehat{CFK}$. This allows us to define the almost $\iota_K$-local equivalence group $\mathfrak{I}^U_K$ of horizontal almost $\iota_K$-complexes, which is equipped with a canonically defined partial order. Then we prove that it is totally ordered modulo the figure-eight knot, from which the definition of the invariant $\mathfrak{A}$ in \Cref{thm:mainthm} follows. 

In \Cref{sec:horfrombordered}, we show that the bordered involution on a 0-framed knot complement naturally induces a horizontal almost $\iota_K$-complex, which is defined uniquely up to almost $\iota_K$-local equivalence. Then, in \Cref{sec:cable}, we prove a part of \Cref{thm:mainthm} which states that $K_{2n+1,1}$ has infinite order in $\mathcal{C}$, by comparing the horizontal $\iota_K$-complex induced by $(-K)_{2n+1,-1}$ with the complex induced by the $(2n+1,-1)$-cable of the figure-eight knot. The comparison is done by passing to the bordered setting and then applying involutive bordered Heegaard Floer theory of knot complements, developed by the first author in \cite{kang2022involutive}. Then we develop the arguments used in its proof to prove \Cref{thm:mainthm-indep}. 

Finally, In \Cref{sec:iteratedcable}, we present a full proof of \Cref{thm:mainthm} by finding an inductive step which allows us to deal with iterated cables. We then use it in \Cref{sec:stronglyrationallyslice} to prove \Cref{thm:mainthm-complexity} by considering the 2-fold cyclic branched double covers of $S^3$ along the given knots.

\subsection*{Questions}
We end this section with few questions. It can be easily verified that for each alternating knot $K$ which is torsion in $\mathcal{C}$, we have that $\mathfrak{A}(K) = \Arf(K)$. Naturally, we ask if the equality holds in general.

\begin{qn}If $K$ is torsion in $\mathcal{C}$, then is $\mathfrak{A}(K)=\Arf(K)$?\end{qn}

Also, since 
the almost local equivalence group of $\iota$-complexes and the almost local equivalence group of knot-like complexes over $\mathbb{F}_2[U,V]/(UV)$ are known to be totally ordered and isomorphic to $\mathbb{Z}^\infty$~\cite{dai2018infinite, DHST:2021-1}, we ask the following question. 
\begin{qn} Is $\mathfrak{I}^U_K$ isomorphic to $\mathbb{Z}^\infty \oplus \mathbb{Z}/2\mathbb{Z}$?\end{qn}

\subsection*{Notation and conventions}
Throughout the paper the unknot is denoted by $O$ and the figure-eight knot is denoted by $E$. Given a knot $K$ and a positive integer $n$, the connected sum of $n$ copies of $K$ is denoted by $nK$, the reverse of the mirror image of $K$ is denoted by $-K$, and the reverse of $K$ is denoted $K^r$. Also, given a closed 3-manifold $Y$, the connected sum of $n$ copies of $Y$ is denoted by $nY$. Lastly, two morphisms (chain maps, type-$D$ morphisms, etc) $f$ and $g$ are homotopic, we write $f\sim g$. In contrast, when two objects $A$ and $B$ are homotopy equivalent, we write $A\simeq B$ instead.

\subsection*{Acknowledgements} The authors wish to thank Matthew Stoffregen, Ian Zemke, and Jae Choon Cha for helpful discussions. The first author is supported by the Institute for Basic Science (IBS-R003-D1). The second author is partially supported by Samsung Science and Technology Foundation (SSTF-BA2102-02) and the POSCO TJ Park Science Fellowship.

\section{The horizontal almost $\iota_K$-local equivalence group}
\label{sec:horizontal}
We briefly recall the theory of involutive knot Floer homology and $\iota_K$-local equivalences. Given an oriented knot $K$ in $S^3$ (or in general, a rationally nullhomologous knot in a rational homology sphere with a prescribed self-conjugate $\spinc$ structure), one starts by representing it by a doubly-pointed Heegaard diagram $H=(\Sigma,\boldsymbol\alpha,\boldsymbol\beta,z,w)$. Note that one can recover $K$ from $H$ in the following way. The given sets of $\alpha$-curves and $\beta$-curves define handlebodies $H_\alpha$ and $H_\beta$ (in which the given sets of curves bound disks, respectively) such that $H_\alpha \cup H_\beta = S^3$. Choose directed paths $K_\alpha \subset H_\alpha$ and $K_\beta \subset H_\beta$, such that $K_\alpha$ goes from $z$ to $w$ and $K_\beta$ goes from $w$ to $z$. Then we have $K=K_\alpha \cup K_\beta$.

Then we consider the $2g$-dimensional smooth manifold $\Sym^g(\Sigma)$, where $g$ denotes the genus of $\Sigma$. This manifold contains two $g$-dimensional tori, $\mathbb{T}_\alpha$ and $\mathbb{T}_\beta$, defined by the product of all $\alpha$-curves and $\beta$-curves, respectively, and two 2-codimensional submanifolds $D_z$ and $D_w$, defined by the images of the natural maps
\[
\begin{split}
\Sym^{g-1}(\Sigma)\times \{z\} &\rightarrow \Sym^g(\Sigma), \\
\Sym^{g-1}(\Sigma)\times \{w\} &\rightarrow \Sym^g(\Sigma).
\end{split}
\]
Then a standard Floer-theoretic argument can be applied to construct the chain complex 
\[
CFK_{UV}(S^3,K):=\left( \bigoplus_{\mathbf{x}\in \mathbb{T}_\alpha \cap \mathbb{T}_\beta} \mathbb{F}_2[U,V] \cdot \mathbf x , \partial \right)
\]
by counting pseudoholomorphic disks of Maslov index 0 with boundary conditions given by $\mathbb{T}_\alpha$ and $\mathbb{T}_\beta$, while recording their algebraic intersection numbers with $D_z$ and $D_w$ by formal variables $U$ and $V$. This complex also comes equipped with a bigrading $(A,M)$ by a pair of integers $A$ and $M$, called the \emph{Alexander} and \emph{Maslov} gradings. It turns out that Heegaard moves induce homotopy equivalences between $CFK_{UV}$, and thus the homotopy type of $CFK_{UV}(S^3,K)$ depends only on the isotopy class of $K$ \cite{ozsvath2004holomorphic,zemke2017quasistabilization}. Furthermore, we have the following localization formula:
\[
(U,V)^{-1}CFK_{UV}(S^3,K) \simeq (U,V)^{-1}\mathbb{F}_2[U,V].
\]

In this paper, we will consider several truncations of the two-variable complex $CFK_{UV}(S^3,K)$. First of all, consider the ring $\mathcal{R}:= \mathbb{F}_2[U,V]/(UV)$. The $\mathcal{R}$-coefficient chain complex 
\[
CFK_{\mathcal{R}}(S^3,K):=CFK_{UV}(S^3,K)\otimes_{\mathbb{F}_2[U,V]} \mathcal{R},
\]
which can be seen as ignoring all ``diagonal'' terms in the differential of $CFK_{UV}(S^3,K)$, will be used in this paper. We will also use the $V=0$ truncation, namely
\[
CFK^-(S^3,K):=CFK_{UV}(S^3,K)\otimes_{\mathbb{F}_2[U,V]} \mathbb{F}_2[U,V]/(V),
\]
and the hat-flavored truncation:
\[
\widehat{CFK}(S^3,K):=CFK_{UV}(S^3,K)\otimes_{\mathbb{F}_2[U,V]} \mathbb{F}_2[U,V]/(U,V).
\]
The above truncations, except for the hat-flavored version, admit localization formulae as well:
\[
\begin{split}
    (U,V)^{-1}CFK_{\mathcal{R}}(S^3,K) &\simeq (U,V)^{-1}\mathcal{R}, \\
    U^{-1}CFK^-(S^3,K) &\simeq \mathbb{F}_2[U,U^{-1}].
\end{split}
\]

Now, given a doubly-pointed Heegaard diagram $H=(\Sigma,\boldsymbol\alpha,\boldsymbol\beta,z,w)$ representing $(S^3,K)$, we consider its conjugate diagram $\overline{H}=(-\Sigma,\boldsymbol\beta,\boldsymbol\alpha,w,z)$, where $-\Sigma$ denotes $\Sigma$ with the opposite orientation. The oriented (and doubly-basepointed) knot represented by $\overline{H}$ is $(K,w,z)$, i.e.\ the same oriented knot $K$ with the positions of basepoints exchanged. Since the holomorphic disk counts in $H$ are the same as those in $\overline{H}$ while the role of $U$ and $V$ are exchanged, the identity map
\[
\conjugate\colon CFK_{UV}(K,z,w)\rightarrow CFK_{UV}(K,w,z)
\]
is a chain homotopy \emph{skew}-equivalence. Choose a self-diffeomorphism $\phi$ of $S^3$ which is supported in a tubular neighborhood of $K$ and exchanges $z$ and $w$. Then it induces a homotopy equivalene
\[
\phi_{\ast}\colon CFK_{UV}(K,w,z)\rightarrow CFK_{UV}(K,z,w).
\]
Then we define $\iota_K=\phi_{\ast}\circ \conjugate$, which is again a homotopy skew-equivalence. Note that $\iota_K$ is defined up to ambiguities given by monodromies along loops of Heegaard moves; since homotopy equivalences induced by those loops are homotopic to identity due to the first-order naturality of Heegaard Floer theory \cite{juhasz2021naturality}, $\iota_K$ is defined uniquely up to homotopy. Since $\iota_K$ exchanges the actions of $U$ and $V$, it descends to $CFK_{\mathcal{R}}$ and $\widehat{CFK}$, but not $CFK^-$.

The involution $\iota_K$ satisfies the following properties, as shown in \cite{hendricks2017involutive} and \cite{zemke2019connected}.
\begin{itemize}
    \item $\iota^2_K \sim 1+\Phi\Psi \sim 1+\Psi\Phi$, where $\Phi$ and $\Psi$ are the formal derivatives of the differential of $CFK_{UV}(S^3,K)$ with respect to $U$ and $V$, respectively.
    \item $\Phi\iota \sim \iota\Psi$.
    \item $\iota$ is skew-graded, i.e. it maps an element of bidegree $(A,M)$ to an element of bidegree $(-A,M)$.
\end{itemize}
Using these properties, the concepts of $\iota_K$-complexes and $\iota_K$-local maps are defined in~\cite{hom2020linear} as follows. A free bigraded chain complex $C$ over $\mathbb{F}_2[U,V]$, together with a chain homotopy skew-equivalence $\iota:C\rightarrow C$, is an \emph{$\iota_K$-complex} if it satisfies the following conditions.
\begin{itemize}
    \item $(U,V)^{-1}C\simeq (U,V)^{-1}\mathbb{F}_2[U,V]$.
    \item $\iota^2 \sim 1+\Phi\Psi$, where $\Phi$ and $\Psi$ are the formal derivatives of the differential of $C$ with respect to $U$ and $V$, respectively.
\end{itemize}
Furthermore, a chain map $f\colon C\rightarrow D$ between $\iota_K$-complexes is an \emph{$\iota_K$-local map} if the following conditions are satisfied.
\begin{itemize}
    \item $(U,V)^{-1}f$ is a homotopy equivalence.
    \item $f\iota \sim \iota f$.
\end{itemize}
If there exists an $\iota_K$-local map from $C$ to $D$ and also from $D$ to $C$, then we say that $C$ and $D$ are \emph{$\iota_K$-locally equivalent} and $f$ is an \emph{$\iota_K$-local equivalence}. 


\begin{rem}
In this paper, we will sometimes use the notion of $\iota_K$-local maps for chain maps between complexes over $\mathcal{R}$. The definition of $\iota_K$-local maps in such a setting is identical to the original definition.
\end{rem}

The concept of almost local $\iota_K$-complexes is also defined in \cite{hom2020linear} by considering instead the hat-flavored involution. 

\begin{defn}
A free bigraded chain complex $C$ over $\mathbb{F}_2[U,V]$, together with a skew-graded homotopy equivalence $\iota\colon\widehat{C}\rightarrow \widehat{C}$, where $\widehat{C}=C\otimes_{\mathbb{F}_2[U,V]}\mathbb{F}_2[U,V]/(U,V)$, is an \emph{almost $\iota_K$-complex} if the following conditions are satisfied.
\begin{itemize}
    \item $(U,V)^{-1}C\simeq (U,V)^{-1}\mathbb{F}_2[U,V]$.
    \item $\iota^2 \sim 1+\Phi\Psi$ as skew-graded chain endomorphisms of $\widehat{C}$, where $\Phi$ and $\Psi$ are the formal derivatives of the differential of $C$ with respect to $U$ and $V$, respectively.
\end{itemize}
Furthermore, a chain map $f\colon C\rightarrow D$ is an \emph{almost $\iota_K$-local map} if the following conditions are satisfied.
\begin{itemize}
    \item $(U,V)^{-1}f$ is a homotopy equivalence.
    \item The induced map $\widehat{f}:\widehat{C}\rightarrow \widehat{D}$, where $\widehat{C}$ and $\widehat{D}$ are the truncations of $C$ and $D$ by $U=V=0$, satisfies $\widehat{f} \iota_K \sim \iota_K \widehat{f}$.
\end{itemize}
If almost $\iota_K$-local maps exist between two almost $\iota_K$-complexes in both directions, we say that they are \emph{almost $\iota_K$-locally equivalent}.
\end{defn}

Note that, given a knot $K$, its knot Floer complex $CFK_{UV}(S^3,K)$, together with the action of $\iota_K$, defines an $\iota_K$-complex. If one only considers the action of $\iota_K$ on $\widehat{CFK}(S^3,K)$, one then obtains an almost $\iota_K$-complex. It is shown in \cite{zemke2019connected} that the (almost) $\iota_K$-local equivalence class of the resulting $\iota_K$-complex (or almost $\iota_K$-complex) depends only on the concordance class of $K$.

Motivated by the definitions of $\iota_K$-complexes and almost $\iota_K$-complexes, we now develop a similar framework by mimicking the structure of $CFK^-(S^3,K)$ of knots $K$ in $S^3$, together with the action of $\iota_K$ on $\widehat{CFK}(S^3,K)$. For simplicity, we make the following definitions. Given a chain complex $C$ over $\mathbb{F}_2[U]$, we say that the chain complex $\widehat{C}=C\otimes_{\mathbb{F}_2[U]} \mathbb{F}_2[U]/(U)$ induced by taking $U=0$ is the \emph{hat-flavored truncation} of $C$. Also, given a chain map $f\colon C\rightarrow D$ between chain complexes over $\mathbb{F}_2[U]$, the induced map $\widehat{f}\colon\widehat{C}\rightarrow \widehat{D}$ between $\widehat{C}$ and $\widehat{D}$ is called the \emph{hat-flavored truncation} of $f$. Sometimes we will confuse $f$ and $\widehat{f}$ to reduce the amount of symbols used.


\begin{defn}
\label{defn:horalmost}
A \emph{horizontal almost $\iota_K$-complex} is a pair $(C,\iota)$ of a bigraded complex $C$ of finitely generated free modules over $\mathbb{F}_2[U]$ and a chain homotopy equivalence $\iota\colon\widehat{C}\rightarrow \widehat{C}$, where $\widehat{C}$ is the hat-flavored truncation of $C$, such that the following conditions are satisfied. 
\begin{itemize}
    \item The formal variable $U$ has bigrading $(-1,-2)$.
    \item $U^{-1}C\simeq \mathbb{F}_2[U,U^{-1}]$
    \item $\iota$ is skew-graded.
    \item $\Phi\iota\Phi\iota \sim \iota\Phi\iota\Phi$.
    \item $\iota^2 \sim 1+\Phi\iota\Phi\iota$.
    \item There exists a chain map $f\colon C\rightarrow C$ whose hat-flavored truncation is homotopic to $\iota\Phi\iota$.
\end{itemize}
Here, $\Phi$ denotes the formal derivative of the differential of $C$ with respect to the formal variable $U$, which is a chain map which is well-defined up to homotopy.
\end{defn}

\begin{rem}
Since we do not have the formal variable $V$ anymore, we had to represent $\Psi$ in terms of $\Phi$ and $\iota$. Luckily, we have $\Phi\iota \sim \iota\Psi$, $\iota^2 \sim 1+\Phi\Psi$, and $\iota^4 \sim \id$, so we replaced $\Psi$ using the formula
\[
\begin{split}
\Psi \sim \iota^{-1}\Phi\iota &= \iota \circ (1+\Psi\Phi)\circ \Phi\iota \\
&= \iota \circ (\Phi + \Psi\Phi^2) \circ \iota \sim \iota\Phi\iota.
\end{split}
\]
Note that, in the last line of the above equation, we used the fact that $\Phi^2$ is nullhomotopic, which follows from the observation that the value of $\Phi$ vanishes on cycles and $\ker(\partial)\subset \ker(\Phi) $.
The last condition in \Cref{defn:horalmost} requires that one can find $\Psi$ as an $\mathbb{F}_2[U]$-linear chain map. This condition is crucial in the proof that the tensor product operation between two horizontal almost $\iota_K$-complexes is commutative.
\end{rem}


\begin{defn}
A degree-preserving chain map $f\colon C\rightarrow D$ between horizontal almost $\iota_K$-complexes $(C,\iota_C)$ and $(D,\iota_D)$ is \emph{almost $\iota_K$-local map} if it satisfies the following conditions.
\begin{itemize}
    \item $f$ is \emph{local}, i.e.\ the localized map $U^{-1}f\colon U^{-1}C\rightarrow U^{-1}D$ is a homotopy equivalence.
    \item $\iota_C \widehat{f} \sim \widehat{f} \iota_D$, where $\widehat{f}$ is the hat-flavored truncation of $f$.
\end{itemize}
Furthermore, if there exist almost $\iota_K$-local maps $f\colon C\rightarrow D$ and $g:D\rightarrow C$, then we say that $(C,\iota_C)$ and $(D,\iota_D)$ are \emph{almost $\iota_K$-locally equivalent}.
\end{defn}

We denote the almost $\iota_K$-local equivalence classes of horizontal almost $\iota_K$-complexes by $\mathfrak{I}^U _K$. We endow this set with a tensor product operation $\otimes$ as $(C,\iota_C)\otimes (D,\iota_D)=(C\otimes D,\iota_{C\otimes D})$, where $\iota_{C\otimes D}$ is defined as
\[
\iota_{C\otimes D} = \iota_C \otimes \iota_D + \Phi\iota_C \otimes \iota_D \Phi .
\]
Note that this operation is modeled after the connected sum formula for involutive knot Floer homology \cite[Theorem 1.1]{zemke2019connected}, which reads
\[
\iota_{K_1 \# K_2} = \iota_{K_1} \otimes \iota_{K_2} + \Phi\iota_{K_1} \otimes \Psi\iota_{K_2}.
\]

\begin{prop}
$(\mathfrak{I}^U _K,\otimes)$ is an abelian group.
\end{prop}
\begin{proof}
We first prove that the group operation is well-defined. Given $C,D\in \mathfrak{I}^U _K$, formally differentiating both sides of
\[
\partial_{C\otimes D} = \partial_C \otimes 1 + 1 \otimes \partial_D
\]
with respect to $U$ gives $\Phi_{C\otimes D}=1\otimes \Phi + \Phi\otimes 1$. From now on we drop the subscripts when the meaning is clear from context.

Thus we have 
\[
\begin{split}
\Phi_{C\otimes D}\iota_{C\otimes D} &\sim (1\otimes \Phi + \Phi \otimes 1)\circ (\iota \otimes \iota + \Phi\iota \otimes \iota\Phi) \\
&\sim \Phi\iota \otimes \iota + \iota \otimes \Phi\iota + \Phi\iota \otimes \Phi\iota\Phi ,
\end{split}
\]
which squares to 
\[
\begin{split}
\Phi_{C\otimes D}\iota_{C\otimes D}\Phi_{C\otimes D}\iota_{C\otimes D} &\sim \Phi\iota\Phi\iota\otimes\iota^2 + \Phi\iota^2 \otimes \iota\Phi\iota+ \Phi\iota\Phi\iota\otimes \iota\Phi\iota\Phi + \iota\Phi\iota\otimes\Phi\iota^2 + \iota^2 \otimes \Phi\iota\Phi\iota + \Phi\iota\Phi\iota\otimes \Phi\iota\Phi\iota \\
&\sim 1\otimes \Phi\iota\Phi\iota + \Phi\iota\Phi\iota\otimes 1 +\Phi\otimes \iota\Phi\iota + \iota\Phi\iota\otimes \Phi.
\end{split}
\]
On the other hand, since 
\[
\begin{split}
\iota_{C\otimes D}\Phi_{C\otimes D} &= (\iota \otimes \iota + \Phi\iota \otimes \iota\Phi) \circ (1\otimes \Phi + \Phi \otimes 1)\\
&\sim \iota\otimes\iota\Phi + \iota\Phi\otimes \iota + \Phi\iota\Phi \otimes \iota\Phi,
\end{split}
\]
we have 
\[
\iota_{C\otimes D}\Phi_{C\otimes D}\iota_{C\otimes D}\Phi_{C\otimes D} \sim 1\otimes \Phi\iota\Phi\iota + \Phi\iota\Phi\iota\otimes 1 +\Phi\otimes \iota\Phi\iota + \iota\Phi\iota\otimes \Phi
\]
as well, which implies that $$\Phi_{C\otimes D}\iota_{C\otimes D}\Phi_{C\otimes D}\iota_{C\otimes D} \sim \iota_{C\otimes D}\Phi_{C\otimes D}\iota_{C\otimes D}\Phi_{C\otimes D}.$$ Also, we have 
\[
\begin{split}
\iota_{C\otimes D}^2 &\sim \iota^2 \otimes \iota^2 + \iota\Phi\iota \otimes \iota^2 \Phi + \Phi\iota^2 \otimes \iota\Phi\iota + \Phi\iota\Phi\iota\otimes \iota\Phi\iota\Phi \\
&\sim 1\otimes 1 + \Phi\iota\Phi\iota\otimes 1 + 1\otimes \Phi\iota\Phi\iota + \Phi\otimes\iota\Phi\iota + \iota\Phi\iota\otimes\Phi \\
&\sim 1+\Phi_{C\otimes D}\iota_{C\otimes D}\Phi_{C\otimes D}\iota_{C\otimes D},
\end{split}
\]
and furthermore, if $f\colon C\rightarrow C$ and $g:D\rightarrow D$ are chain maps such that their hat-flavored truncations $\widehat{f}$ and $\widehat{g}$ are homotopic to $\iota\Phi\iota$, then we have
\[
\begin{split}
    \iota_{C\otimes D}\Phi_{C\otimes D}\iota_{C\otimes D} &= \iota\Phi\iota\otimes \iota^2 + \iota\Phi^2 \iota\otimes \iota^2\Phi + \Phi\iota\Phi\iota \otimes \iota\Phi\iota + \Phi\iota\Phi^2\iota\otimes \iota\Phi\iota\Phi \\  &\quad + \iota^2 \otimes \iota\Phi\iota + \iota\Phi\iota \otimes \iota\Phi\iota\Phi + \Phi\iota^2 \otimes \iota\Phi^2 \iota + \Phi\iota\Phi\iota \otimes \iota\Phi^2 \iota\Phi \\
    &\sim \iota\Phi\iota \otimes 1 + 1\otimes\iota\Phi\iota,
\end{split}
\]
which is represented by the chain map $f\otimes 1+1\otimes g$. So we deduce that $C\otimes D$ is also an element of $\mathfrak{I}^U _K$, and thus the group operation is well-defined.

Next, we show that the group operation is associative. Choose any $A,B,C\in\mathfrak{I}^U _K$. Then we have 
\[
\begin{split}
\iota_{A\otimes(B\otimes C)} &\sim \iota\otimes(\iota\otimes \iota + \Phi\iota\otimes\iota\Phi) + \Phi\iota \otimes (\iota\Phi \otimes \iota + \iota\otimes \iota\Phi) \\ 
&\sim \iota\otimes\iota\otimes\iota + \iota+\Phi\iota+\iota\Phi+\Phi\iota\otimes\iota\Phi\otimes\iota + \Phi\iota\otimes\iota\otimes\iota\Phi.
\end{split}
\]
On the other hand, we have 
\[
\begin{split}
\iota_{(A\otimes B)\otimes C} &\sim (\iota\otimes\iota+\Phi\iota\otimes\iota\Phi)\otimes \iota + (\Phi\iota\otimes\iota + \iota\otimes\Phi\iota)\otimes \iota\Phi \\ 
&\sim \iota\otimes\iota\otimes\iota+\Phi\iota\otimes\iota\Phi\otimes\iota + \Phi\iota\otimes\iota\otimes\iota\Phi + \iota\otimes\Phi\iota\otimes\iota\Phi.
\end{split}
\]
Hence $\iota_{A\otimes(B\otimes C)} \sim \iota_{(A\otimes B)\otimes C}$, and thus the group operation is associative.

Now we show that every element of $\mathfrak{I}^U _K$ admits an inverse. Given any $C\in\mathfrak{I}^U _K$, endow its dual complex $C^{\ast}$ with the dual involution $\iota^{\ast}$, and note that $\Phi_{C^{\ast}}=\Phi^{\ast}_C$. Consider the trace map $\tr\colon C\otimes C^{\ast} \rightarrow \mathbb{F}[U]$ and its dual $\cotr\colon\mathbb{F}[U]\rightarrow C\otimes C^{\ast}$. Then, by the tensor-hom adjunction, we have 
\[
\begin{split}
\tr\circ \iota_{C\otimes C^{\ast}} &\sim \tr\circ (\iota \otimes \iota^{\ast} + \Phi\iota \otimes \iota^{\ast}\Phi^{\ast}) \\ 
&\sim \tr \circ (\iota^2 \otimes 1 + \Phi\iota\Phi\iota \otimes 1) \sim \tr
\end{split}
\]
and also 
\[
\begin{split}
\iota_{C\otimes C^{\ast}} \circ \cotr &\sim (\iota \otimes \iota^{\ast} + \Phi\iota \otimes \iota^{\ast}\Phi^{\ast}) \circ \cotr\\ 
&\sim (\iota^2 \otimes 1 + \Phi\iota\Phi\iota \otimes 1) \circ \cotr \sim \cotr.
\end{split}
\]
Hence $C\otimes C^{\ast}$ is almost $\iota_K$-localy trivial, and hence $C^{\ast}$, endowed with $\iota^{\ast}$, is the inverse of $C$ in $\mathfrak{I}^U_K$. Thus $\mathfrak{I}^U _K$ is a group.

It remains to show that the group operation is commutative. Given two elements $C,D$ in $\mathfrak{I}^U_K$, consider two horizontal almost $\iota_K$-complexes $(C\otimes D,\iota_1)$ and $(C\otimes D,\iota_2)$, where $\iota_1$ and $\iota_2$ are defined as 
\[
    \iota_1 = \iota\otimes \iota + \Phi\iota \otimes \iota\Phi \qquad\text{and}\qquad
    \iota_2 = \iota\otimes \iota + \iota\Phi \otimes \Phi\iota,
\]
respectively. We claim that they are almost $\iota_K$-locally equivalent. To prove this, define an $\mathbb{F}_2[U]$-linear chain map $f\colon C\otimes D\rightarrow C\otimes D$ by 
\[
F = \id\otimes \id + f \otimes \Phi,
\]
where $f\colon C\rightarrow C$ is a chain map whose hat-flavored truncation $\widehat{f}$ is homotopic to $\iota\Phi\iota$. Then $F$ satisfies $U^{-1}F \sim \id$ since 
\[
F^2 = \id\otimes \id + f^2 \otimes \Phi^2 \sim \id\otimes \id,
\]
 and we also have 
\[
    \iota_1 F = \iota\otimes \iota + \Phi\iota \otimes \iota\Phi + \iota^2 \Phi\iota \otimes \iota\Phi + \Phi\iota^2 \Phi\iota \otimes \iota\Phi^2 \sim \iota\otimes \iota
\]
since $\iota^2 \sim 1+\Phi\iota\Phi\iota \sim 1+\iota\Phi\iota\Phi$ and $\Phi^2 \sim 0$. Similarly, we also have 
\[
F\iota_2 = \iota\otimes \iota + \iota\Phi \otimes \Phi\iota + \iota\Phi\iota^2 \otimes \Phi\iota + \iota\Phi\iota^2 \Phi \otimes \Phi^2 \iota \sim \iota\otimes\iota .
\]
So $\iota_1 F + F \iota_2 \sim 0$, and thus $F$ is an almost $\iota_K$-local map from $(C\otimes D,\iota_1)$ to $(C\otimes D,\iota_2)$. However the same argument also shows that $F$ also works as an almost $\iota_K$-local map from $(C\otimes D,\iota_2)$ to $(C\otimes D,\iota_1)$. Therefore we deduce that $\mathfrak{I}^U _K$ is an abelian group.
\end{proof}

\begin{rem}
\label{rem:homomorphism}
As in the case of (almost) $\iota_K$-local maps, given a knot $K$, its minus-flavored knot Floer complex $CFK^-(S^3,K)$, together with the action of $\iota_K$ on $\widehat{CFK}(S^3,K)$, defines an element of $\mathfrak{I}^U_K$, which depends only on the concordance class of $K$. This gives a well-defined map
\[
\mathfrak{A}:\mathcal{C}\rightarrow \mathfrak{I}^U_K.
\]
It follows from \cite{zemke2019connected} that $\mathfrak{A}$ is a group homomorphism.
\end{rem}

We define a partial order on the group $\mathfrak{I}^U _K$ in the following way. Given two elements $C$ and $D$ of $\mathfrak{I}^U _K$, we define $C \le D$ if there exists an almost $\iota_K$-local map from $C$ to $D$. It is clear that the relation $\le$ on $\mathfrak{I}^U_K$ is indeed a partial order, since $C_1 \le C_2$ implies $C_1 \otimes D \le C_2 \otimes D$ for any $D\in \mathfrak{I}^U_K$. If neither $C \le D$ nor $D \le C$ holds, we say that $C$ and $D$ are \emph{incomparable}.

\begin{exmp}
The trivial complex $C_O$ 
is defined as the trivial complex $\mathbb{F}_2[U]$, whose generator 1 has bidegree $(0,0)$, together with the trivial involution on its hat-flavored truncation. It is clear that $C_O$ represents the identity element of $\mathfrak{I}^U_K$ and this complex is represented by the involutive knot Floer homology of the unknot $O$.
\end{exmp}

\begin{exmp}
\label{exmp:figureeight}
Consider the chain complex $C_E$, generated over $\mathbb{F}_2[U]$ by elements $a,b,c,d,x$, where $a$ and $x$ has bidegree $(0,0)$. The differential is given by 
\[
\partial a=Ub,\,\partial c=Ud,\,\partial b=\partial d=\partial x=0.
\]
Its hat-flavored truncation $\widehat{C}_E$ is still generated by $a,b,c,d,x$, but has zero differential. We define an involution $\iota_E$ on $\widehat{C}_E$ by
\[
\iota_E(a)=a+x,\,\iota_E(b)=c,\,\iota_E(c)=b,\,\iota_E(d)=d,\,\iota_E(x)=x+d.
\]
Then $C_E=(C_E,\iota_E)$ is an element of $\mathfrak{I}^U_K$. Note that this complex is represented by the involutive knot Floer homology of the figure-eight knot $E$. 

To see that $C_E$ is incomparable to $C_O$ in $\mathfrak{I}^U_K$, suppose on the contrary that it is. Since the figure-eight knot is negative-amphichiral, we have $-[C_E]=[C_E] \in \mathfrak{I}^U_K$. Therefore, we should have an almost $\iota_K$-local map $f\colon C_E\rightarrow C_O$. Since $x$ generates the homology of $U^{-1}C_E$, we should have $f(x)=1$. However, since $\widehat{C}_E$ has zero differential and $C_O$ has trivial involution, we should also have
\[
f(x)=f(a+\iota_E(a)) = f(a)+f(\iota_E(a)) = f(a)+\iota_{O}(f(a)) = 0,
\]
a contradiction. Therefore $C_E$ and $C_O$ are incomparable.
\end{exmp}

\Cref{exmp:figureeight} already shows that $\mathfrak{I}^U_K$ is not totally ordered. However, we will see that, after taking quotient by the order 2 subgroup generated by $C_E$, the induced partial order on the quotient group is a total ordering. In other words, $\mathfrak{I}^U_K$ is totally ordered modulo $C_E$.

\begin{lem}
\label{lem:figeightmap}
Let $(C,\iota)$ be a horizontal almost $\iota_K$-complex. Suppose that there exists an element $a^{\prime}$ and a cycle $x^{\prime}$ in $C$ such that the homology class of $x^{\prime}$ generates the homology of $U^{-1}C$ and $a^{\prime}+\iota(a^{\prime})=x^{\prime}$ in $\widehat{C}$. Then there exists an almost $\iota_K$-local map from $C_E$ to $C$.
\end{lem}
\begin{proof}
We may assume without loss of generality that the differential $\partial$ of $C$ is reduced, i.e.\ the induced differential on $\widehat{C}$ is zero. Then there exists an element $b^{\prime}\in C$ such that $\partial a^{\prime}=Ub^{\prime}$. Choose a lift $c^{\prime} \in C$ of $\iota(b^{\prime})\in \widehat{C}$, and choose again an element $d^{\prime} \in C$ such that $\partial c^{\prime} = Ud^{\prime}$. Then we have in $\widehat{C}$ that 
\[
d^{\prime} = \Phi\iota\Phi\iota(a^{\prime}) = (1+\iota^2)(a^{\prime}) = (1+\iota)^2(a^{\prime}) = x^{\prime} + \iota(x^{\prime}).
\]
Therefore the map from $C_E$ to $C$ which maps $a,b,c,d,x$ to $a^{\prime},b^{\prime},c^{\prime},d^{\prime},x^{\prime}$, respectively, is an almost $\iota_K$-local equivalence.
\end{proof}

\begin{thm}
\label{thm:comparison}
Two elements $C,D \in \mathfrak{I}^U_K$ are incomparable if and only if $[C]=[D]+[C_E]$.
\end{thm}
\begin{proof}
Let $C,D \in \mathfrak{I}^U_K$. First, suppose that $[C]=[D]+[C_E]$, then we have [$C\otimes D^{\ast}] = [C_E]$. Since $C_O$ and $C_E$ are incomparable by \Cref{exmp:figureeight}, we conclude that $C$ and $D$ are also incomparable.

For the other direction, we will show that for each $C \in \mathfrak{I}^U_K$, if there is no almost $\iota_K$-local map from $C_E$ to $C$, then there is an almost $\iota_K$-local map from $C$ to $C_O$. This is enough since if $C$ and $D$ are incomparable, then $C\otimes D^{\ast}$ and $C_O$ are incomparable, which means that there exist no almost $\iota_K$-local maps to $C_O$ from both $C\otimes D^{\ast}$ and $C^{\ast}\otimes D$. Then it follows from the assertion that there exist almost $\iota_K$-local maps from $C_E$ to them. Since $[C_E]=-[C_E]$, we deduce that $C\otimes D^{\ast}$ is almost $\iota_K$-locally equivalent to $C_E$, and thus $[C]=[D]+[C_E]$. 
	
Possibly after replacing $C$ by a homotopy equivalent complex, we may assume that there exists an $\mathbb{F}_2[U]$-linear basis $B = \{x, y_i, z_i\}$ for $C$, so that for each $i$, we have $\partial y_i = U^{n_i} z_i$ for some $n_i>0$. For simplicity, we will write $1+\iota$ as $\omega$. Consider the $\mathbb{F}_2$-vector spaces 
$$Z = \spa_{\mathbb{F}_2} \{ z_i \} \qquad\text{and}\qquad W = \spa_{\mathbb{F}_2} \{\omega (g) \mid g \in B\}$$
which are linear subspaces of $\widehat{C}$. We claim that $x$ is not contained in $Z + W$. Indeed, suppose it were. Then $x=z+w$ where $z$ is a torsion cycle and $w\in \image(\omega)$, which implies that there exists an element $b$ in the $\mathbb{F}_2$-linear span of $B$ such that $\omega(b)=x+z$. But now the homology class of $x+z$ generates $H_{\ast}(U^{-1}C)$, so it follows from \Cref{lem:figeightmap} that there exists an almost $\iota_K$-local map from $C_E$ to $C$, a contradiction. The claim follows.

	Now choose a homogeneous $\mathbb{F}_2$-basis $\{p_1, \ldots, p_r\}$ for $Z + W$. Extend the set $\{x, p_1, \ldots, p_r\}$ to a homogenous $\mathbb{F}_2$-basis
	\[
	\{x, p_1, \ldots, p_r, q_1, \ldots, q_s\}
	\]
	for all of $\spa_{\mathbb{F}_2}(B)$. Define
	\[
	C^{\prime} = \spa_{\mathbb{F}_2[U]}\{p_1, \ldots, p_r, q_1, \ldots, q_s\}.
	\]
	Since the image of $\partial$ is contained in the $\mathbb{F}_2[U]$-linear span of $Z$, it is clear that $C^{\prime}$ is a subcomplex of $C$. Then the quotient map
	\[
	C \rightarrow C/C^{\prime} \simeq \mathbb{F}_2[U]
	\]
	is a local map from $C$ to $C_O$, completing the proof.
\end{proof}

\Cref{thm:comparison} then allows us to define an additive $\mathbb{Z}/2\mathbb{Z}$-valued invariant on the torsion subgroup of $\mathcal{C}$ via involutive knot Floer homology, as follows.

\begin{cor}
\label{cor:pseudoArf}
The group homomorphism $\mathfrak{A}\colon\mathcal{C}\rightarrow \mathfrak{I}^U_K$ in \Cref{rem:homomorphism} restricts to a homomorphism $$\mathfrak{A}\colon\mathcal{C}_{T}\rightarrow \mathbb{Z}/2\mathbb{Z},$$ where $\mathcal{C}_{T}$ denotes the torsion subgroup of the knot concordance group $\mathcal{C}$ and $\mathbb{Z}/2\mathbb{Z}$ is the subgroup of $\mathfrak{I}^U_K$ generated by the chain complex $C_E$.
\end{cor}
\begin{proof}
We need to prove that if $K$ is torsion in $\mathcal{C}$, then  either $\mathfrak{A}(K)=[C_O]$ or $\mathfrak{A}(K)=[C_E]$ in $\mathfrak{I}^U_K$. Suppose not; then it follows from \Cref{thm:comparison} that $\mathfrak{A}(K)$ is comparable to $C_O$. Assume without loss of generality that $\mathfrak{A}(K) \geq C_O$. Then we get
\[
[C_O] \le \mathfrak{A}(K) \le 2\mathfrak{A}(K) \le \cdots \le n\mathfrak{A}(K)=\mathfrak{A}(nK) = [C_O],
\]
which implies $\mathfrak{A}(K)=[C_O]$, a contradiction.
\end{proof}

Furthermore, the existence of a total ordering on $\mathfrak{I}^U_K$ modulo $C_E$ gives a new method to determine whether the given knot has infinite order in the concordance group.

\begin{cor}
\label{cor:infiniteorder}
If $\mathfrak{A}(K)$ is comparable to both $C_O$ and $C_E$, then $K$ has infinite order in $\mathcal{C}$.
\end{cor}
\begin{proof}
If $K$ is torsion in $\mathcal{C}$, then it follows from \Cref{cor:pseudoArf} that $\mathfrak{A}(K)=[C_O]$ or $\mathfrak{A}(K)=[C_E]$ in $\mathfrak{I}^U_K$. This is a contradiction, since we know from \Cref{exmp:figureeight} that $C_O$ and $C_E$ are not comparable to each other.
\end{proof}

\begin{exmp}
Note that $\mathfrak{A}(O)=0$ and $\mathfrak{A}(E)=1$ by definition. In general, it follows directly from \cite[Proposition 8.3]{hendricks2017involutive} that for any quasi-alternating torsion knot $K$, we have that $\mathfrak{A}(K)=\Arf(K)$.
\end{exmp}

Finally, we will see that the invariant $\mathfrak{A}$ is insensitive under orientation reversal for torsion knots, although its definition involves involutive knot Floer homology and thus takes the orientation on the given oriented knot into account in an essential way. This fact will be used in \Cref{sec:iteratedcable}.

\begin{prop}
\label{prop:reversal}
For any torsion knot $K$, we have $\mathfrak{A}(K^r)=\mathfrak{A}(K)$.
\end{prop}
\begin{proof}
It suffices to show that $\mathfrak{A}(K)=0$ implies $\mathfrak{A}(K^r)=0$. Suppose that $\mathfrak{A}(K)=0$. As observed in \cite{binns2021nonorientable}, the involution $\iota_{K^r}$ on $CFK_{UV}(S^3,K^r) = CFK_{UV}(S^3,K)$ is the homotopy inverse of $\iota_K$. By assumption, there exists an $\iota_K$-local map
\[
f\colon C_O\rightarrow CFK^-(S^3,K),
\] 
i.e.\ $f$ is a local map whose hat-flavored truncation $\widehat{f}$ satisfies $\widehat{f} \sim \iota_K \widehat{f}$. Composing both sides with $\iota_{K^r}$ then gives 
\[
\widehat{f} \sim \iota_{K^r}\iota_K \widehat{f} \sim \iota_{K^r} \widehat{f},
\]
and thus $f$ is actually an $\iota_K$-local map from $C_O$ to $CFK^-(S^3,K^r)$. Hence $\mathfrak{A}(K^r)$ is comparable to $C_O$. Since $K$ is torsion in $\mathcal{C}$, we see that $K^r$ should also be torsion. By \Cref{cor:pseudoArf} and \Cref{cor:infiniteorder}, we should have $\mathfrak{A}(K^r)=0$, as desired.
\end{proof}

\section{Horizontal almost $\iota_K$-complexes from bordered Floer homology}
\label{sec:horfrombordered}
Bordered Heegaard Floer homology, defined by Lipshitz, Ozsvath, and Thurston in \cite{lipshitz2018bordered}, associates to a bordered 3-manifold $M$ (with connected boundary $\Sigma$) a type-$D$ module $\widehat{CFD}(M)$ and a type-$A$ module $\widehat{CFA}(M)$ over the $\mathbb{F}_2$-algebra $\mathcal{A}(\Sigma)$, which are defined via counting holomorphic disks on any choice of bordered Heegaard diagram of $M$ with respect to certain boundary conditions. When we have to specify the diagram we are working on, we will denote them as $\widehat{CFD}(H)$ and $\widehat{CFA}(H)$ where $H$ is our choice of a bordered Heegaard diagram representing $M$. The categories of type-$D$ (or type-$A$) modules also have a notion for morphisms between them, which are the $\mathcal{A}(\Sigma)$-linear maps which commute with the differential in the type-$D$ case and the maps which are compatible with the $A_\infty$ relations in the type-$A$ case. Note that we are not giving precise definitions here; in general, we expect the readers to be familiar with bordered Heegaard Floer theoretic machinaries.

When $M$ has two boundary components, say $\Sigma_1$ and $\Sigma_2$, we instead have a type-$DA$ module $\widehat{CFDA}(M)$, which is a $\mathcal{A}(\Sigma_1)$-$\mathcal{A}(\Sigma_2)$-bimodule. This bimodule is also defined using bordered Heegaard diagrams representing $M$, and we will denote its type-$DA$ module as $\widehat{CFDA}(H)$ when we have to specify a Heegaard diagram $H$ that we are working with. As in the case of type-$D$ and type-$A$ modules, type-$DA$ modules also have a notion of morphisms between them, which are called type-$DA$ morphisms.

A similar construction also works when we have a knot inside a bordered 3-manifold, in which case we use doubly-pointed bordered Heegaard diagram such that one basepoint lies on the boundary. Such diagrams give minus-flavored type-$D$ and type-$A$ modules over $\mathbb{F}_2[U]$, denoted $CFD^-$ and $CFA^-$, in which algebraic intersection numbers of pseudoholomorphic disks with the interior basepoint is recorded by a formal variable $U$. The truncations of those modules by $U=0$ are denoted again by $\widehat{CFD}$ and $\widehat{CFA}$. Note that the truncation by $U=1$ is equivalent to forgetting the knot.

This framework is compatible with gluing constructions in the following way. If $M$ and $N$ are bordered 3-manifolds with connected boundary and $\partial M = \partial N$, then we have
\[
\widehat{CF}(M \cup N) \simeq \widehat{CFA}(M) \boxtimes \widehat{CFD}(N).
\]
Similar formulae also holds for gluing bordered manifolds with one or two boundary components, and also when we have a knot inside either $M$ or $N$ (and then the formula involves $CFA^-$ or $CFD^-$). The general rule is that a type-$D$ boundary should be glued to a type-$A$ boundary.

When bordered manifolds have boundary components diffeomorphic to torus, the associated $\mathbb{F}_2$-algebra $\mathcal{A}(T^2)$ admits a simple description, as follows. As an $\mathbb{F}_2$-vector space, we have
\[
\mathcal{A}(T^2) = \mathrm{Span}_{\mathbb{F}_2}(\iota_0,\iota_1,\rho_1,\rho_2,\rho_3,\rho_{12},\rho_{23},\rho_{123}),
\]
where $\iota_0$ and $\iota_1$ are idempotents which sum up to 1. We will mainly focus on the case of knot complements, i.e.\ $S^3 \smallsetminus K$, in which case its $\widehat{CFD}$ can be computed directly from the $\mathcal{R}$-coefficient knot Floer chain complex of $K$. We will also make use of bordered manifolds with two torus boundaries, which arise as complements of pattern knots in the $\infty$-framed solid torus, denoted $T_\infty$.

We now briefly recall the involutive theory for bordered Heegaard Floer theory, as defined in \cite{hendricks2019involutive}. Let $M$ be a bordered 3-manifold with a torus boundary. Choose a bordered Heegaard diagram $H=(\Sigma,\boldsymbol\alpha,\boldsymbol\beta,z)$ representing $M$ and its conjugate diagram $\overline{H}=(-\Sigma,\boldsymbol\beta,\boldsymbol\alpha,z)$. Since the holomorphic disk counts for $H$ and $\overline{H}$ are identical, the identity maps give a canonical isomorphism
\[
    \conjugate: \widehat{CFD}(\overline{H}) \rightarrow \widehat{CFD}(H).
\] 
Unlike involutive Heegaard Floer homology of closed 3-manifolds, the conjugate diagram $\overline{H}$ does not represent $M$, but rather represents $\AZ \cup M$, where $\AZ$ denotes the Auroux-Zarev dualizing piece, which is represented by the bordered diagram in \Cref{fig:AZpiece}. Choosing a sequence of Heegaard moves from $\AZ \cup H$ to $\overline{H}$ gives a homotopy equivalence between $\widehat{CFDA}(\AZ)\boxtimes \widehat{CFD}(H)$ and $\widehat{CFD}(\overline{H})$, and composing it with $\conjugate$ gives a homotopy equivalence
\[
\iota_M\colon\widehat{CFDA}(\AZ)\boxtimes \widehat{CFD}(M)\rightarrow \widehat{CFD}(M).
\]

Note that, due to the lack of naturality in bordered Heegaard Floer homology, the homotopy equivalence $\iota_M$ might not be defined uniquely up to homotopy, although this ambiguity has almost no effect on our arguments involving $\iota_M$. We denote the set of all possible maps $\iota_M$ that can possibly arise from its definition by $\Inv_D(M)$. The elements of $\Inv_D(M)$ are called \emph{bordered involutions} of $M$.

\begin{figure}[hbt]
    \centering
    \includegraphics[width=0.3\textwidth]{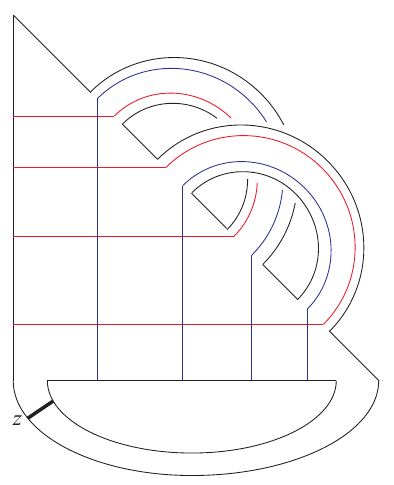}
\caption{A $\alpha$-$\beta$-bordered Heegaard diagram representing the Auroux-Zarev piece $\AZ$.}
\label{fig:AZpiece}
\end{figure} 

One can also define bordered involutions for bordered 3-manifolds with two torus boundaries. Given such a bordered 3-manifold $N$ and an $\alpha$-$\beta$- bordered Heegaard diagram $H$ representing it, we can similarly define the conjugate diagram $\overline{H}$ of $H$. In this case, it is shown in \cite[Theorem 4.6]{lipshitz2011heegaard} that $\overline{H}$ represents $\AZ\cup N \cup \overline{\AZ}$, where $\overline{\AZ}$ denotes the mirror of $\AZ$. Proceeding as in the one torus boundary case, one can define the set $\Inv(N)$ of \emph{bordered involutions} $\iota_N$ of $N$, which are homotopy equivalences of the form
\[
\iota_N \colon \widehat{CFDA}(\AZ)\boxtimes \widehat{CFDA}(N) \boxtimes \widehat{CFDA}(\overline{\AZ}) \rightarrow \widehat{CFDA}(N).
\] 

Recall that, given a bordered 3-manifold $N$ and $M$, where $N$ has two torus boundaries and $M$ has one torus boundary, we have a pairing formula
\[
\widehat{CFD}(N \cup M) \simeq \widehat{CFDA}(N)\boxtimes \widehat{CFD}(M).
\]
Given any choices of bordered involutions $\iota_N \in \Inv(N)$ and $\iota_M \in \Inv_D(M)$, we can describe a bordered involution of the glued manifold $N \cup M$ as follows. Consider the homotopy equivalence
\[
\begin{split}
\iota_{\iota_N,\iota_M}\colon &\widehat{CFDA}(\AZ)\boxtimes \widehat{CFD}(N\cup M) \\
&\simeq \widehat{CFDA}(\AZ)\boxtimes \widehat{CFDA}(N) \boxtimes \widehat{CFD}(M) \\
&\simeq \widehat{CFDA}(\AZ)\boxtimes \widehat{CFDA}(N)\boxtimes \widehat{CFDA}(\overline{\AZ}) \boxtimes\widehat{CFDA}(\AZ)\boxtimes \widehat{CFD}(M) \\
&\xrightarrow{\iota_N \boxtimes \iota_M} \widehat{CFDA}(N)\boxtimes \widehat{CFD}(M) \simeq \widehat{CFD}(N\cup M).
\end{split}
\]
Note that, in the construction of $\iota_{\iota_N,\iota_M}$, we used a homotopy equivalence 
\[
\Omega \colon  \widehat{CFDA}(\mathbb{I}) \xrightarrow{\simeq} \widehat{CFDA}(\overline{\AZ}) \boxtimes\widehat{CFDA}(\AZ) ,
\]
where $\mathbb{I}$ denotes the trivial cylinder $T^2 \times [0,1]$, induced from the observation that $\overline{\AZ}\cup \AZ$ is diffeomorphic to $\mathbb{I}$. It is shown in \cite[Lemma 4.4]{hendricks2019involutive} that $\widehat{CFDA}(\mathbb{I})$ is homotopy-rigid, so that $\Omega$ is uniquely defined up to homotopy, and hence the homotopy class of $\iota_{\iota_N,\iota_M}$ depends only on the homotopy classes of $\iota_N$ and $\iota_M$. Following the proof of \cite[Theorem 5.1]{hendricks2019involutive}, we can immediately see that $\iota_{\iota_N,\iota_M}$ is a bordered involution for $N \cup M$, so that we have a well-defined map
\[
\begin{split}
    \Inv(N) \times \Inv_D(M) &\rightarrow \Inv_D(N\cup M), \\
    (\iota_N , \iota_M) &\mapsto \iota_{\iota_N,\iota_M}.
\end{split}
\]

Now, given a torsion knot $K$, its horizontal almost $\iota_K$-complex $\mathfrak{A}(K)$ can be algebraically recovered from the involutive bordered Floer homology of $S^3 \smallsetminus K$ as follows. Consider the $\infty$-framed solid torus $T_\infty$ and the longitudinal knot $\nu \subset T_\infty$. Then we have a pairing formula:
\[
CFK^-(S^3,K) \simeq CFA^-(T_\infty,\nu) \boxtimes \widehat{CFD}(S^3 \smallsetminus K),
\]
where $S^3 \smallsetminus K$ is 0-framed. Its hat-flavored truncation is given by 
\[
\widehat{CFK}(S^3,K) \simeq \widehat{CFA}(T_\infty,\nu) \boxtimes \widehat{CFD}(S^3,K).
\]
Choose any bordered involution $\iota_{S^3 \smallsetminus K} \in \Inv_D(S^3 \smallsetminus K)$. Then, as shown (in a type-$D$ version) in \cite[Theorem 1.3]{kang2022involutive}, there exists a type-$A$ morphism 
\[
\mathfrak{f}^A_{dual}\colon\widehat{CFA}(T_\infty,\nu) \rightarrow \widehat{CFA}(T_\infty,\nu)\boxtimes\widehat{CFDA}(\AZ)
\]
such that the map
\[
\begin{split}
    \widehat{CFA}(T_\infty,\nu) \boxtimes \widehat{CFD}(S^3,K) &\xrightarrow{\mathfrak{f}^A_{dual} \boxtimes \id} \widehat{CFA}(T_\infty,\nu) \boxtimes \widehat{CFDA}(\AZ) \boxtimes \widehat{CFD}(S^3,K) \\
    &\xrightarrow{\id\boxtimes \iota_{S^3 \smallsetminus K}} \widehat{CFA}(T_\infty,\nu) \boxtimes \widehat{CFD}(S^3,K)
\end{split}
\]
is homotopic to either $\iota_K$ and $\iota^{-1}_K$. Since it follows from the proof of \cite[Theorem 1.3]{kang2022involutive} that the above map is either always homotopic to $\iota_K$ for all $K$ or always homotopic to $\iota^{-1}_K$ for all $K$, we will assume that it is actually always homotopic to $\iota_K$ for all knots $K$. Note that, if it were homotopic to $\iota^{-1}_K$ instead, all proofs work in the same way with a slight modification.

Now we will prove two lemmas, \Cref{lem:conjugatemorphism} and \Cref{lem:almostbordered}, which describe how to relate type-$D$ morphisms between bordered Floer homology of 0-framed knot complements to $\iota_K$-local maps between horizontal almost $\iota_K$-complexes. These lemmas are crucial in the proofs of the main theorems.

\begin{lem}
\label{lem:conjugatemorphism}
Given two knots $K_1,K_2$ and a type-$D$ morphism $f\colon\widehat{CFD}(S^3 \smallsetminus K_1) \rightarrow \widehat{CFD}(S^3 \smallsetminus K_2)$, denote by $\mathhat(f)$ the induced map
\[
\begin{split}
    \widehat{CFK}(S^3,K_1) &\simeq \widehat{CFA}(T_\infty,\nu)\boxtimes \widehat{CFD}(S^3 \smallsetminus K_1) \\
    &\xrightarrow{\id\boxtimes f} \widehat{CFA}(T_\infty,\nu) \boxtimes \widehat{CFD}(S^3 \smallsetminus K_2) \simeq \widehat{CFK}(S^3,K_2).
\end{split}
\]
Then, for any choices of $\iota_{S^3 \smallsetminus K_i} \in \Inv_D(S^3 \smallsetminus K_i)$ for $i=1,2$, we have that 
\[
\mathhat\left(\iota_{S^3 \smallsetminus K_2} \circ (\id\boxtimes f)\circ \iota^{-1}_{S^3 \smallsetminus K}\right) \sim \iota_{K_2} \circ \mathhat(f) \circ \iota^{-1}_{K_1},
\]
where $\iota_{K_i}$ denotes the involutive actions on $\widehat{CFK}(S^3,K_i)$ for $i=1,2$.
\end{lem}
\begin{proof}
Observe that 
\[
\begin{split}
    \left(\id \boxtimes \left(\iota_{S^3 \smallsetminus K_2} \circ (\id\boxtimes f)\circ \iota^{-1}_{S^3 \smallsetminus K_1}\right)\right) \circ \left(\id\boxtimes \iota_{S^3 \smallsetminus K_1}\right) \circ \left(\mathfrak{f}^A_{dual} \boxtimes \id\right) &\sim  
    \left(\id \boxtimes \left(\iota_{S^3 \smallsetminus K_2} \circ (\id\boxtimes f)\right)\right) \circ \left(\mathfrak{f}^A_{dual} \boxtimes \id\right) \\
    &\sim \left(\id \boxtimes \iota_{S^3 \smallsetminus K_2}\right) \circ \left(\mathfrak{f}^A_{dual} \boxtimes \id\right) \circ (\id \boxtimes f).
\end{split}
\]
Therefore we deduce that
\[
\mathhat\left(\iota_{S^3 \smallsetminus K_2} \circ (\id\boxtimes f)\circ \iota^{-1}_{S^3 \smallsetminus K_1}\right) \circ \iota_{K_1}
\sim \iota_{K_2} \circ \mathhat(f),
\]
which implies $$\mathhat\left(\iota_{S^3 \smallsetminus K_2} \circ (\id\boxtimes f)\circ \iota^{-1}_{S^3 \smallsetminus K_1}\right) \sim \iota_{K_2} \circ \mathhat(f) \circ \iota^{-1}_{K_1},$$ as desired.
\end{proof}

\begin{defn}
Given two knots $K_1,K_2$, we say that a type-$D$ morphism $f\colon\widehat{CFD}(S^3 \smallsetminus K_1)\rightarrow \widehat{CFD}(S^3 \smallsetminus K_2)$ is \emph{local} if the induced map
\[
\begin{split}
    \widehat{CF}(S^3) &\simeq \widehat{CFA}(T_\infty)\boxtimes \widehat{CFD}(S^3 \smallsetminus K_1) \\
    &\xrightarrow{\id\boxtimes f} \widehat{CFA}(T_\infty) \boxtimes \widehat{CFD}(S^3 \smallsetminus K_2) \simeq \widehat{CF}(S^3)
\end{split}
\]
is a homotopy equivalence. Furthermore, a local morphism $f\colon\widehat{CFD}(S^3 \smallsetminus K_1) \rightarrow \widehat{CFD}(S^3 \smallsetminus K_2)$ is said to be \emph{almost bordered-$\iota$-local} if $f$ is local and $$\mathhat\left(f+\iota_{S^3 \smallsetminus K_2} \circ (\id\boxtimes f)\circ \iota^{-1}_{S^3 \smallsetminus K_1}\right) \sim 0$$ for some choices of $\iota_{S^3 \smallsetminus K_i} \in \Inv_D(S^3 \smallsetminus K_i)$ for $i=1,2$.
\end{defn}

In fact, the definition of bordered-$\iota$-locality does not depend on the choices of $\iota_{S^3 \smallsetminus K_i}$ by the following lemma.

\begin{lem}
\label{lem:almostbordered}
Given two knots $K_1,K_2$ and a local morphism $f\colon\widehat{CFD}(S^3 \smallsetminus K_1) \rightarrow \widehat{CFD}(S^3 \smallsetminus K_2)$, denote by $\mathminus(f)$ the induced map 
\[
\begin{split}
    CFK^- (S^3,K_1) &\simeq CFA^-(T_\infty,\nu)\boxtimes \widehat{CFD}(S^3 \smallsetminus K_1) \\
    &\xrightarrow{\id\boxtimes f} CFA^-(T_\infty,\nu) \boxtimes \widehat{CFD}(S^3 \smallsetminus K_2) \simeq CFK^-(S^3,K_2).
\end{split}
\]
Then for any choices of $\iota_{S^3 \smallsetminus K_i} \in \Inv_D(S^3 \smallsetminus K_i)$ for $i=1,2$, the map $\mathminus(f)$ is an almost $\iota_K$-local map if and only if $f$ is  almost bordered-$\iota$-local.
\end{lem}
\begin{proof}
Since $f$ is local, we know that $\mathminus(f)$ is also a local map, since $\widehat{CFA}(T_\infty)$ is the truncation of $CFA^-(T_\infty,\nu)$ by $U=1$. By \Cref{lem:conjugatemorphism}, we have 
\[
\begin{split}
    \mathhat\left(f+\iota_{S^3 \smallsetminus K_2} \circ (\id\boxtimes f)\circ \iota^{-1}_{S^3 \smallsetminus K_1}\right) &\sim \mathhat(f) + \iota_{K_2} \circ \mathhat(f) \circ \iota_{K_1}^{-1} \\
    &\sim \left(\mathhat(f) \circ \iota_{K_1} + \iota_{K_2} \circ \mathhat(f)\right) \circ \iota_{K_1}^{-1}.
\end{split}
\]
Since $\mathhat(f)$ is the hat-flavored truncation of $\mathminus(f)$, we deduce that $\mathminus(f)$ is an almost $\iota_K$-local map if and only if $$\iota_{K_2} \circ \mathhat(f) \sim \mathhat(f) \circ \iota_{K_1},$$ and if and only if $$\mathhat\left(f+\iota_{S^3 \smallsetminus K_2} \circ (\id\boxtimes f)\circ \iota^{-1}_{S^3 \smallsetminus K_1}\right) \sim 0,$$
which concludes the proof.
\end{proof}

\section{$(\text{odd},1)$-cables of torsion knots}
\label{sec:cable}

In this section, we will prove a part of \Cref{thm:mainthm} which states that if $K$ is torsion in $\mathcal{C}$ with $\mathfrak{A}(K) = 1$, then $K_{2n+1,1}$ has infinite order in $\mathcal{C}$, and then improve its proof a bit more to prove \Cref{thm:mainthm-indep}. We start by observing that the $\mathcal{R}$-coefficient knot Floer chain complex of a torsion knot has a $\mathcal{R}$ summand.

\begin{lem}
\label{lem:torsionsplitting}
If $K$ is torsion in $\mathcal{C}$, then $CFK_{\mathcal{R}}(S^3,K)$ admits a direct summand isomorphic to $\mathcal{R}$, which is concentrated at bidegree $(0,0)$.
\end{lem}
\begin{proof}
The local equivalence group $\mathfrak{K}$ of knot-like complexes, defined in~\cite{DHST:2021-1}, is endowed with a total ordering. In particular, $\mathfrak{K}$ is torsion-free. Hence we deduce that $CFK_{\mathcal{R}}(S^3,K)$ is locally trivial, which implies that there exist almost local maps 
\[
f\colon\mathcal{R}\rightarrow CFK_{\mathcal{R}}(S^3,K) \qquad\text{and}\qquad\,g:CFK_{\mathcal{R}}(S^3,K)\rightarrow \mathcal{R}.
\]
Since $g\circ f$ is a local map from $\mathcal{R}$ to itself, it should be homotopic to the identity map. Therefore $f$ and $g$ induce a splitting
\[
CFK_{\mathcal{R}}(S^3,K)\simeq \mathcal{R}\oplus A
\]
for some $\mathcal{R}$-acyclic summand $A$.
\end{proof}

Recall that, for a torsion knot $K$, the value $\mathfrak{A}(K)$ is 0 if the horizontal almost $\iota_K$-complex $CFK^-(K)$ is almost $\iota_K$-locally equivalent to the $C_O$, and $1$ if it is almost $\iota_K$-locally equivalent to $C_E$. Since the notion of almost $\iota_K$-local equivalence involves chain maps which are not $\mathcal{R}$-linear, but only $\mathbb{F}_2[U]$-linear, it is not quite compatible with the bordered techniques. However, due to \Cref{lem:torsionsplitting}, we can say a bit more than that, and as a result, we have the following lemma.

\begin{lem}
\label{lem:borderedinterpret}
If $K$ is torsion in $\mathcal{C}$ with $\mathfrak{A}(K)=1$, then there exists another knot $K^{\prime}$, concordant to $K$, such that $\widehat{CFD}(S^3 \smallsetminus K^{\prime})$ has a direct summand isomorphic to $\widehat{CFD}(S^3 \smallsetminus E)$. Furthermore, the inclusion map $i\colon \widehat{CFD}(S^3 \smallsetminus E) \rightarrow \widehat{CFD}(S^3 \smallsetminus K^{\prime})$ is almost bordered-$\iota$-local.
\end{lem} 
\begin{proof}
Let $n$ be a positive integer such that $nK$ is slice. Then $2n(K\# E)= 2nK\# 2nE$ is slice, so $K\# E$ also is torsion in $\mathcal{C}$. By \Cref{lem:torsionsplitting}, we see that  $CFK_{\mathcal{R}}(S^3,K\# E)$ admits a direct summand isomorphic to $\mathcal{R}$, concentrated at bidegree $(0,0)$; denote its generator as $x$. Since $\mathfrak{A}(K\# E)=\mathfrak{A}(K)+\mathfrak{A}(E)=0$, the knot Floer complex of $K\# E$ is almost $\iota_K$-locally equivalent to $C_O$, so there exists an almost $\iota_K$-local map $f\colon C_O\rightarrow CFK^-(S^3,K\# E)$. 
Write $f(1)=x+c$ for some cycle $c$ which represents a $U$-torsion class in homology. Then we have
\[
\partial y = (\widehat{f}\circ \iota_{C_O} + \iota_{K\# E} \circ \widehat{f})(1) = x+c+\iota_{K\# E}(x+c)
\]
for some $y\in \widehat{CFK}(S^3,K\# E)$. Choose any lifts to $CFK_{\mathcal{R}}(S^3,K)$ of $c$ and $y$, and denote them again by $c$ and $y$. Then we get
\[
\iota_{K\# E}(x) = \iota_{K\# E}(x+c)+\iota_{K\# E}(c) = x+c+\iota_{K\# E}(c)+\partial y \pmod{U,V},
\]
Assume without loss of generality that the differential on $CFK_{\mathcal{R}}(S^3,K\# E)$ is reduced, i.e.\ its truncation by $U=V=0$ is zero. Then, for some $\alpha,\beta\in CFK_{\mathcal{R}}(S^3,K\# E)$, we can write
\[
\iota_{K\# E}(x)=x+c+\iota_{K\# E}(c)+U\alpha+V\beta.
\]
Also, we can decompose the differential $\partial$ of $CFK_{\mathcal{R}}(S^3,K\# E)$ as 
\[
\partial = \partial_U + \partial_V,
\]
where $\partial_U$ and $\partial_V$ consists of terms containing powers of $U$ and $V$, respectively.

Recall that $c$ is a lift of a cycle in $CFK^-(S^3,K\# E)$. Since $UV=0$ in $\mathcal{R}$, we should have $\partial_U (c)=0$. Since $\iota_{K\# E}$ is a chain skew-map, we have $\partial_U \circ \iota_{K\# E} = \iota_{K\# E} \circ \partial_V$. Thus we get
\[
\begin{split}
    \partial_V (c) &= \partial_V (x+\iota_{K\# E}(x)) + \partial_V(\iota_{K\# E}(c)) + U\partial_V(\alpha) + V\partial_V(\beta) \\
    &=\iota_{K\# E}(\partial_U(c)) + V\partial_V(\beta) = V\partial_V(\beta).
\end{split}
\]
Hence, if we define $c^{\prime}=c+V\beta$, then we have
\[
\begin{split}
    \partial_U(c^{\prime}) &= V\partial_U(\beta) = 0, \\
    \partial_V(c^{\prime}) &= \partial_V(c)+V\partial_V(\beta) = 0,
\end{split}
\]
so $c^{\prime}$ is a cycle. Since $x$ generates a free $\mathcal{R}$ summand of $CFK_{\mathcal{R}}(S^3,K\# E)$ and the homology class of $c^{\prime}$ is $\mathcal{R}$-torsion, we deduce that $x+c^{\prime}$ also generates a free summand. Furthermore, since $c^{\prime} = c \pmod{U,V}$, we have 
\[
\iota_{K\# E}(x+c^{\prime}) = \iota_{K\# E}(x+c) = x+c \pmod{U,V}.
\]

Now take $K^{\prime} := K \# 2E$. Since $x+c^{\prime}$ generates a free summand of $CFK_{\mathcal{R}}(S^3,K\# E)$, we see that $(x+c^{\prime})\otimes CFK_{\mathcal{R}}(S^3,E)$ is a direct summand of $CFK_{\mathcal{R}}(S^3,K^{\prime})$. It follows from \cite[Theorem 11.26]{lipshitz2011heegaard} that direct summands of $CFK_{\mathcal{R}}$ of a knot induce direct summands of $\widehat{CFD}$ of its 0-framed complement, so there exists a direct summand of $\widehat{CFD}(S^3 \smallsetminus K^{\prime})$ isomorphic to $\widehat{CFD}(S^3 \smallsetminus E)$ such that the inclusion map $i\colon \widehat{CFD}(S^3 \smallsetminus E) \rightarrow \widehat{CFD}(S^3 \smallsetminus K^{\prime})$ satisfies the property that $\mathrm{Minus}(i)$ is an almost $\iota_K$-local map.
\end{proof}

For $p>1$, denote the $(2n+1,-1)$-cabling pattern by $C_p$. The type-$A$ module $\widehat{CFA}(T_\infty,C_p)$ was computed in \cite[Lemma 8.3]{Oz-St-Sz:2017-1}. To summarize, it admits a model generated by an $\iota_0$-element $z$ and $\iota_1$-elements $b_1,\cdots,b_{2n+1},a_1,\cdots,a_{2n+1}$. The $A_\infty$-operations are given as follows.
\[
\xymatrix{
& a_{2n+1} \ar[ld]_{\rho_2}\ar[dd]^{U^{2n+1}} && a_{2p}\ar[ll]_{\rho_2,\rho_1}\ar[dd]^{U^{2p}} && \cdots\ar[ll]_{\rho_2,\rho_1} && a_2\ar[ll]_{\rho_2,\rho_1}\ar[dd]^{U^2} && a_1\ar[ll]_{\rho_2,\rho_1} \ar[dd]^U\\
z \ar[rd]_{\rho_3} & && && && && \\
& b_{2n+1} && b_{2p}\ar[ll]^{U\cdot (\rho_2,\rho_1)} && \cdots\ar[ll]^{U\cdot (\rho_2,\rho_1)} && b_2\ar[ll]^{U\cdot (\rho_2,\rho_1)} && b_1\ar[ll]^{U\cdot (\rho_2,\rho_1)}
}
\]
Recall from \cite[Proposition 3.1]{hom2020linear} that $CFK_{UV}(S^3,E_{2n+1,-1})$ has generators 
\[
a,b,c,d,e,f,g,b_{0,i},\cdots,g_{0,i},b_{1,i},\cdots,g_{1,i},b_{0,p},\cdots,f_{0,p},b_{1,p},\cdots,g_{1,p}.
\]
for $1\le i\le p-1$. Among them, $a$ generates a direct summand and $b,c,d,e,f,g$ generate the ``central'' summand, satisfying $\iota_{E_{2n+1,-1}}(b)=b+a \pmod{U,V}$ and $\iota_{E_{2n+1,-1}}(a)=a \pmod{U,V}$. 

It follows from \cite[Theorem 11.26]{lipshitz2018bordered} that the type-$D$ module $\widehat{CFD}(S^3 \smallsetminus E)$ is generated by 9 elements 
\[
e_0,f_0,g_0,h_0,e_1,f_1,g_1,h_1,w,
\]
where the differential is given as follows.
\[
\xymatrix{
f_0 \ar[d]_{\rho_1} & e_1 \ar[l]_{\rho_2} & e_0 \ar[l]_{\rho_3} \ar[d]^{\rho_1} & & \\
f_1 & & h_1 & \oplus & w \ar@(ur,dr)^{\rho_{12}} \\
g_0 \ar[u]^{\rho_{123}} & g_1\ar[l]^{\rho_2} & h_0\ar[l]^{\rho_3} \ar[u]_{\rho_{123}} & &
}
\]
We consider the elements of $CFK^-(S^3,E_{2n+1,-1})$ which correspond to $z\otimes w$, $z\otimes e_0$, and $b_{2p+i} \otimes h_1$ under the pairing formula
\[
CFK^-(S^3,E_{2n+1,-1}) \simeq CFA^-(T_\infty,\nu) \boxtimes \widehat{CFD}(S^3 \smallsetminus E).
\]
\begin{lem}
\label{lem:elementb}
The element $b$ can be written mod $U$ as an $\mathbb{F}_2$-linear combination of $a_1\otimes h_1,\cdots,a_{2n+1}\otimes h_1$.
\end{lem}
\begin{proof}
It follows from the $A_\infty$ operations of $CFA^-(T_\infty,\nu)$ that $\partial(a_i \otimes h_1) = U^i b_i\otimes h_1$. Also, $b_i \otimes h_1$ are cycles in $HFK^-(S^3,E_{2n+1,-1})$, which are not boundaries since $\rho_3 h_1$ is not a cycle. Furthermore, it is clear that $z\otimes w$ and $a_i\otimes h_1$ have bidegree $(0,0)$. Comparing this with computations in \cite{hom2020linear}, we deduce that the $\mathbb{F}_2$-linear span of those elements should contain either $b+a+U\gamma$ or $b+U\gamma$ for some cycle $\gamma$. After a basis change by $b\mapsto b+a$, we may assume without loss of generality that the span actually contains $b+U\gamma$.
\end{proof}

\begin{lem}
\label{lem:cabledmap}
Let $K$ be a knot such that $CFK_{\mathcal{R}}(S^3,K)$ admits a direct summand isomorphic to $\mathcal{R}$, which is concentrated at bidegree $(0,0)$. If $f\colon\widehat{CFD}(S^3 \smallsetminus E)\rightarrow \widehat{CFD}(S^3 \smallsetminus K)$ is a type-$D$ morphism with $\mathhat(f) \sim 0$, then the map
\[
\begin{split}
f_{2n+1,-1}\colon\widehat{CFD}(S^3 \smallsetminus E_{2n+1,-1}) &\simeq \widehat{CFDA}(T_\infty \smallsetminus C_p) \boxtimes \widehat{CFD}(S^3 \smallsetminus E) \\
&\xrightarrow{\id\boxtimes f} \widehat{CFDA}(T_\infty \smallsetminus C_p) \boxtimes \widehat{CFD}(S^3 \smallsetminus K) \simeq \widehat{CFD}(S^3 \smallsetminus K_{2n+1,-1})
\end{split}
\]
satisfies $\mathhat(f_{2n+1,-1})(a_i \otimes h_1) =0$ for each $1 \le i \le 2n+1$.
\end{lem}
\begin{proof}
We may assume without loss of generality that the tensored differential obtained by box-tensoring the type-$D$ modules of 0-framed knot complements with $\widehat{CFA}(T_\infty,\nu)$ is zero. Then the condition $\mathhat(f) \sim 0$ becomes $\mathhat(f)=0$. Since $\widehat{CFA}(T_\infty,\nu)$ has one $\iota_0$-generator, no $\iota_1$-generators, and no nontrivial $A_\infty$ operations, it follows that $f(w)=\sum_{b\in B} \rho_b b$, where $B$ is a basis of $\widehat{CFD}(S^3 \smallsetminus K)$ and $\rho_b\in \{\rho_1,\rho_2,\rho_3,\rho_{12},\rho_{23},\rho_{123}\}$. Suppose that $\mathhat(f_{2n+1,-1})(a_i \otimes h_1) \ne 0$ for some $1 \le i \le 2n+1$. Recall from \cite[Theorem 11.26]{lipshitz2018bordered} that $\widehat{CFD}$ of a knot complement consists of stable chains and unstable chains, together with possible linear changes of basis along the endpoints of stable chains. Due to the assumption that $CFK_{\mathcal{R}}(S^3,K)$ admits an $\mathcal{R}$ summand, we should have a splitting
\[
\widehat{CFD}(S^3 \smallsetminus K) \simeq \widehat{CFD}(S^3 \smallsetminus O) \oplus N
\]
and $N$ consists only of stable chains. Denote by $t$ the generator of $\widehat{CFD}(S^3 \smallsetminus O)$, so that $\partial t=\rho_{12}t$. Suppose that $f(h_1)$, which is an element of $\widehat{CFD}(S^3 \smallsetminus K)$, is not contained in the $N$ summand. Denote by $p_u$ the projection
\[
p_u\colon\widehat{CFD}(S^3 \smallsetminus K) \rightarrow \widehat{CFD}(S^3 \smallsetminus O).
\]
Then we should have $p_u(f(h_1))=\rho_2 t$. But then we get
\[
\rho_{12}p_u(f(e_0))=\partial p_u(f(e_0)) = p_u(f(\partial e_0)) = \rho_3 p_u(f(e_1))+\rho_{12} w.
\]
Since $\rho_{12}w$ is not a multiple of $\rho_3$, it follows that $p_u(f(e_0))=w$, which would imply that $\mathhat(f)$ cannot be nullhomotopic, a contradiction. Hence we should have $f(h_1)\in N$.

Recall again that $N$ consists of stable chains. Denote the sets of its $\iota_0$-generators and $\iota_1$-generators by $B_0$ and $B_1$, respectively. Since $h_1$ is a cycle and all nontrivial $A_\infty$-operations of $\widehat{CFA}(T_\infty,C_p)$ starting from $a_i$ begins with $\rho_2$, in order to have $\mathhat(f_{2n+1,-1})(a_i \otimes h_1)\ne 0$, we should have 
\[
f(h_1) = \rho_2 \cdot \sum_{\alpha \in S_0} \alpha + (\text{terms involving }\iota_1\text{-generators})
\]
for some nonempty subset $S_0 \subset B_0$. Indeed, assume that $S_0 = \emptyset$. Then we can write 
\[
f(h_1) = \sum_{\beta \in S_1} \beta + \rho_{23} \cdot \sum_{\beta \in T_1} \beta 
\]
for some $S_1,T_1 \subset B_1$ with $S_1 \cap T_1 = \emptyset$. Since $h_1$ is a cycle and no nontrivial $A_\infty$-operations of $\widehat{CFA}(T_\infty,C_p)$ starts with $\rho_{23}$, we should have $S_1 \ne \emptyset$. Since $\partial e_0 = \rho_3 e_1 + \rho_1 h_1$, we have 
\[
\partial f(e_0) = \rho_3 f(e_1) + \rho_1 \sum_{\beta \in S_1} \beta + \rho_{123} \sum_{\beta \in T_1} \beta.
\]
Since $N$ consists only of stable chains, the above equation implies that each $\beta \in S_1$ must be contained in a stable chain of the form
\[
a_\beta \xrightarrow{\rho_1} \beta \xleftarrow{\rho_{23}}\cdots\xleftarrow{\rho_{123}} a^{\prime}_{\beta}
\]
and $f(e_0)$ contains $a_\beta$. But then, if $s$ denotes the generator of $\widehat{CFA}(T_\infty,\nu)$, the image under $\mathhat(f)$ of $e_0 \otimes s$ must contain $a_\beta \otimes s$, which implies that $\mathhat(f)(e_0 \otimes s)$ is a cycle in $\widehat{CFK}(S^3,K)$ whose homology class is nontrivial. Since $e_0 \otimes s$ is a cycle of $\widehat{CFK}(S^3,E)$, this is a contradiction. Hence $\mathhat(f)$ is not nullhomotopic, a contradiction.

Since each $\alpha \in S_0$ is either the startpoint or the endpoint of a stable chain, we should have either $\partial (\rho_2 \alpha) = \rho_{23} b_{\alpha}$ for some $b_{\alpha}\in B_1$ or $\rho_2 \alpha = \partial b_{\alpha}$ for some $b_{\alpha} \in B_1$. However, if $\rho_2 \alpha = \partial b_{\alpha}$, then since $h_1$ is a cycle, one can eliminate $\rho_2 \alpha$ from $f(h_1)$ by homotoping $f$. Hence we may assume without loss of generality that we do not have such basis elements in $S_0$. Then we must have 
\[
0 = f(\partial h_1) = \partial f(h_1) = \rho_{23} \cdot \sum_{\alpha \in S_0} b_{\alpha} + \partial(\text{terms involving }\iota_1\text{-generators}).
\]
Furthermore, for any $\iota_1$-generator $\beta \in B_1$, its differential $\partial \beta$ is always of the form $\rho_{2}\gamma$ or $\rho_{23} \gamma$, and the latter terms should cancel out the $\rho_{23}b_\alpha$ terms in the above equation. However, since each $b_\alpha$ is in a stable chain of the form 
\[
\alpha \xrightarrow{\rho_3} b_\alpha \xrightarrow{\rho_{23}}\cdots\xrightarrow{\rho_{2}} c_\alpha,
\]
there is no $\iota_1$-generator whose differential can cancel out $\rho_{23}b_\alpha$. Hence no cancellation can happen, which means that we should have $S_0=\emptyset$, a contradiction. Therefore $\mathhat(f_{2n+1,-1})(a_i \otimes h_1) =0$.
\end{proof}

We are finally ready to prove that $K_{2n+1,1}$ has infinite order in $\mathcal{C}$.

\begin{thm}
\label{thm:odd1hasinfiniteorder}
If $K$ is torsion in $\mathcal{C}$ with $\mathfrak{A}(K) = 1$, then for any nonzero integer $n$ the cable $K_{2n+1,1}$ has infinite order in $\mathcal{C}$.
\end{thm}
\begin{proof}
Note that $K_{2n+1,1}$ has infinite order in $\mathcal{C}$ if and only if $-K_{2n+1,1} = (-K)_{2n+1,-1}$ has infinite order in $\mathcal{C}$. Moreover, since $\mathfrak{A}(K) = \mathfrak{A}(-K)$, it is enough to show that $K_{2n+1,-1}$ has infinite order in $\mathcal{C}$.

By \Cref{lem:almostbordered} and \Cref{lem:borderedinterpret}, we may assume that there exists a type-$D$ morphism $f\colon\widehat{CFD}(S^3 \smallsetminus E) \rightarrow \widehat{CFD}(S^3 \smallsetminus K)$ such that, for any choice of bordered involutions $\iota_{S^3 \smallsetminus E}\in \Inv_D(S^3 \smallsetminus E)$ and $\iota_{S^3 \smallsetminus K} \in \Inv_D(S^3 \smallsetminus K)$, we have $$\mathhat\left(f + \iota_{S^3 \smallsetminus K} \circ f \circ \iota^{-1}_{S^3 \smallsetminus E}\right) \sim 0.$$ Then, by \Cref{lem:elementb} and \Cref{lem:cabledmap}, we have 
\[
\mathhat\left(f_{2n+1,-1} + \left(\iota_{S^3 \smallsetminus K} \circ \left(\id\boxtimes f\right) \circ \iota^{-1}_{S^3 \smallsetminus E}\right)_{2n+1,-1}\right)(b) = \mathhat\left(\left(f + \iota_{S^3 \smallsetminus K} \circ \left(\id\boxtimes f\right) \circ \iota^{-1}_{S^3 \smallsetminus E}\right)_{2n+1,-1}\right)(b) =0.
\]
Since it follows from the description of bordered involutions of glued manifolds that
\[
\left(\iota_{S^3 \smallsetminus K} \circ (\id\boxtimes f) \circ \iota^{-1}_{S^3 \smallsetminus E}\right)_{2n+1,-1} \sim \iota_{S^3 \smallsetminus K_{2n+1,-1}} \circ (\id\boxtimes f_{2n+1,-1}) \boxtimes \iota^{-1}_{S^3 \smallsetminus E_{2n+1,-1}},
\]
we deduce from \Cref{lem:conjugatemorphism} that
\[
\begin{split}
    &\mathhat\left(f_{2n+1,-1}\right)(b)+\left(\iota_{K_{2n+1,-1}} \circ \mathhat(f_{2n+1,-1}) \circ \iota^{-1}_{E_{2n+1,-1}}\right)(b) \\
    &= \mathhat\left(f_{2n+1,-1} + \iota_{S^3 \smallsetminus K_{2n+1,-1}} \circ (\id\boxtimes f_{2n+1,-1}) \boxtimes \iota^{-1}_{S^3 \smallsetminus E_{2n+1,-1}}\right)(b) =0.
\end{split}
\]
Since $\iota^{-1}_{E_{2n+1,-1}}(b)=b+a \pmod{U,V}$ on $CFK_{\mathcal{R}}(S^3,E_{2n+1,-1})$, we get
\[
\iota^{-1}_{K_{2n+1,-1}}\circ\mathhat(f_{2n+1,-1})(b)=\mathhat(f_{2n+1,-1})(b+a).
\]
Denote the differentials of minus-flavored knot Floer complexes by $\partial^-$. Then $\partial^- b$ is a multiple of $U^n$, so $\partial^- \mathminus(f_{2n+1,-1})(b)$ should also be a multiple of $U^n$. Since $n>1$, we have $\Phi\circ \mathminus(f_{2n+1,-1})(b) = 0 \pmod{U}$, which implies that $\Phi\circ \mathhat(f_{2n+1,-1})(b)=0$. Hence we have
\[
\begin{split}
\iota_{K_{2n+1,-1}}\circ \mathhat(f_{2n+1,-1})(b) &= \iota^{-1} _{K_{2n+1,-1}} \circ \left(1+\iota_{K_{2n+1,-1}}\Phi\iota_{K_{2n+1,-1}}\Phi\right)\circ \mathrm{Hat}(f_{2n+1,-1})(b) \\
&= \iota^{-1} _{K_{2n+1,-1}}\circ\mathhat(f_{2n+1,-1})(b) \\
&= \mathhat(f_{2n+1,-1})(b+a).
\end{split}
\]

Furthermore, since $a$ is a cycle which generates the homology of $U^{-1}CFK^-(S^3,E_{2n+1,-1})$, its image $\mathminus(f_{2n+1,-1})(a)$ should also be a cycle which generates the homology of $U^{-1}CFK^-(S^3,K_{2n+1,-1})$. It then follows from \Cref{lem:figeightmap} that $C_E \le \mathfrak{A}(K_{2n+1,-1})$ in $\mathfrak{I}^U_K$.

Now, we have 
\[
\begin{split}
\iota_{K_{2n+1,-1}}\circ \mathhat(f_{2n+1,-1})(a) &= \iota_{K_{2n+1,-1}} \circ \left(\mathhat(f_{2n+1,-1})(b) + \iota_{K_{2n+1,-1}}\circ \mathhat(f_{2n+1,-1})(b)\right) \\
&= \mathhat(f_{2n+1,-1})(a)+\iota_{K_{2n+1,-1}}\Phi\iota_{K_{2n+1,-1}}\Phi\circ\mathhat(f_{2n+1,-1})(b) \\
&=  \mathhat(f_{2n+1,-1})(a).
\end{split}
\]
Hence the inclusion of the subcomplex of $CFK^-(S^3,K_{2n+1,-1})$ generated by $\mathminus(f_{2n+1,-1})(a)$ is an almost $\iota_K$-local map by \Cref{lem:almostbordered}. This implies that $C_O \geq \mathfrak{A}(K_{2n+1,-1})$. Therefore it follows from \Cref{cor:infiniteorder} that $C_O<\mathfrak{A}(K_{2n+1,-1})$ and hence $K_{2n+1,-1}$ has infinite order in $\mathcal{C}$.
\end{proof}


To prove \Cref{thm:mainthm-indep}, we have to approximate $CFK^-(S^3,K_{2n+1,-1})$ by simpler horizontal almost $\iota_K$-complexes. To do so, for each $n>1$, we consider an element $C_n$ of $\mathfrak{I}^U_K$, generated by elements $a_n,b_n,c_n,d_n,x_n$, where $a_n$ and $x_n$ have bidegree $(0,0)$. The differential is given by
\[
\partial a_n = U^n b_n,\,\partial c_n = U^n d_n, \, \partial b_n = \partial d_n = \partial x_n = 0
\]
and the involution (on its hat-flavored truncation $\widehat{C}_n$) is given by
\[
\iota_{C_n}(a_n)=a_n+x_n,\,\iota_{C_n}(b_n) = c_n,\,\iota_{C_n}\text{ fixes } d_n,x_n.
\]
Since $\Phi=0$ on $\widehat{C}_n$, it is clear that $C_n$ is a well-defined horizontal almost $\iota_K$-complex. Furthermore, since the chain map $f_n:C_n\rightarrow C_{n+1}$ defined by 
\[
f_n(a_n)=a_{n+1},\,f_n(b_n)=Ub_{n+1},\, f_n(c_n)=f_n(d_n)=0,\, f_n(x_n)=x_{n+1}
\]
satisfies $f_n \circ \iota = \iota \circ f_n$ on $\widehat{C}_n$, and it is clear that there exists an almost $\iota_K$-local map from $C_O$ to $C_n$ but not from $C_n$ to $C_O$, it follows from \Cref{thm:comparison} that 
\[
0 < C_2 \le C_3 \le C_4 \le \cdots.
\]
We will see that the complexes $C_n$ provide both lower and upper bounds for $CFK^-(S^3,K_{2n+1,-1})$.

\begin{lem}
\label{lem:upperbound}
Let $K$ be a knot such that $CFK^-(S^3,K)$ admits an $\mathbb{F}_2[U]$ direct summand, generated by an element $x$ of bidegree $(0,0)$. If $n$ is a positive integer such that $U^n$ annihilates the torsion submodule of $HFK^-(S^3,K)$, then we have that $\mathfrak{A}(K) < C_{n+1}$. 
\end{lem}
\begin{proof}
Suppose that there exists an almost $\iota_K$-local map $f\colon C_{n+1}\rightarrow CFK^-(S^3,K)$. By \Cref{lem:torsionsplitting}, $CFK_{\mathcal{R}}(S^3,K)$ admits a splitting
\[
CFK_{\mathcal{R}}(S^3,K) \simeq A \oplus \mathbb{F}_2[U]\cdot x
\]
for an acyclic summand $A$. Choose a basis $\{y_1,\cdots,y_k,z_1,\cdots,z_k\}$ of $A$ which is \emph{horizontally simplified}, i.e.\ $\partial y_i = U^{n_i} z_i$ for some $n_i \le n$ and $\partial z_i = 0$ for all $i=1,\cdots ,k$. Since $\partial a_{n+1}=U^{n+1}b_{n+1}$ in $C_{n+1}$, we get
\[
\partial f(a_{n+1}) = f(\partial a_{n+1})=U^{n+1}f(b_{n+1}).
\]
Since $f(b_{n+1})$ is a cycle whose homology class is torsion, we have
\[
f(b_{n+1}) = \sum_{i\in S_z} z_i
\]
for some $S_z\subset \{1,\cdots,k\}$. Write 
\[
f(a_{n+1}) = z+ \sum_{i\in S_y} U^{c_i} y_i
\]
for some $S_y \subset \{1,\cdots,k\}$, $c_i>0$, and $z\in \spa_{\mathbb{F}_2[U]}(x,z_1,\cdots,z_k)$. Then we get
\[
U^{n+1}\sum_{i\in S_z} z_i = U^{n+1}f(b_{n+1}) = \partial f(a_{n+1}) = \sum_{i\in S_y} U^{n_i+c_i}z_i,
\]
which implies that $S_y = S_z$ and $n_i + c_i = n+1$ for all $i \in S_y$. But this means $c_i = n+1-n_i \ge 1$ for all $i$, so we deduce that $f(a_{n+1})=z \pmod U$. Since $\iota(x_{n+1})=x_{n+1}$, we have
\[
0 = (\iota f + f \iota)(a_{n+1}) = f(a_{n+1})+\iota(f(a_{n+1}))+f(x_{n+1}) = f(x_{n+1})+z+\iota(z).
\]
Furthermore, since $U^{-1}f$ is a homotopy equivalence, the homology class of $f(x_{n+1})$ generates $U^{-1}H_{\ast}(C)$. Hence $x$ should appear when we write $\iota(z)$ as an $\mathbb{F}_2$-linear combination of the given basis elements. 

Recall that the involution $\iota$ on $\widehat{CFK}(S^3,K)$ is a truncation of an involution on $\widehat{CFK}(S^3,K)$ which is a homotopy skew-equivalence. Assume without loss of generality that the differential on $CFK_{\mathcal{R}}(S^3,K)$ is reduced, and write
\[
\partial = \partial_U + \partial_V
\]
where $\partial_U$ and $\partial_V$ are the sums of terms in $\partial$ involving $U$-powers and $V$-powers, respectively. Since $z$ is a $\partial_U$-cycle whose homology class is $U$-torsion, $\iota(z)$ should be the truncation from $CFK_{\mathcal{R}}(S^3,K)$ of a $\partial_V$-cycle whose homology class is $V$-torsion. In particular, $\iota(z)$ cannot contain $x$, as $x$ comes from a cycle whose homology class is not annihilated by any powers of $U$ and $V$, a contradiction. So we conclude that there does not exist an almost $\iota_K$-local map from $C_{n+1}$ to $CFK^-(S^3,K)$.

By \Cref{thm:comparison}, we are now left with two possibilities: $\mathfrak{A}(K) < [C_{n+1}]$ or $\mathfrak{A}(K)+[C_E] = [C_{n+1}]$. However, $CFK^-(S^3,K)\otimes E \simeq CFK^-(S^3,K\# E)$ also satisfies the assumptions of the lemma, so there is no almost $\iota_K$-local map from $C_{n+1}$ to $CFK^-(S^3,K)\otimes C_E$. Therefore we conclude that $\mathfrak{A}(K) < [C_{n+1}]$.
\end{proof}

We are now ready to prove \Cref{thm:mainthm-indep}, whose statement we recall.
\begin{thm}
\label{thm:mainthm-indep-body}
If $K$ is torsion in $\mathcal{C}$ with $\mathfrak{A}(K) = 1$, then the set of cables $\{K_{2n+1,1}\}_{n>0}$ contains an infinite subset which is linearly independent in $\mathcal{C}$.
\end{thm}
\begin{proof}As in the proof of \Cref{thm:odd1hasinfiniteorder}, it is enough to show that $\{K_{2n+1,-1}\}_{n>0}$ contains an infinite subset which is linearly independent in $\mathcal{C}$. Also, recall that, in the proof of \Cref{thm:odd1hasinfiniteorder}, we constructed a local chain map 
\[
\mathminus(f_{2n+1,-1})\colon CFK^-(S^3,E_{2n+1,-1})\rightarrow CFK^-(S^3,K_{2n+1,-1})
\]
for each $n>1$, such that its hat-flavored truncation $\mathhat(f_{2n+1,-1})$ satisfies
\[
\begin{split}
    \iota_{K_{2n+1,-1}}\circ\mathhat(f_{2n+1,-1})(b) &= \mathhat(f_{2n+1,-1})(b+a), \\
    \iota_{K_{2n+1,-1}}\circ\mathhat(f_{2n+1,-1})(a) &= \mathhat(f_{2n+1,-1})(a).
\end{split}
\]
Since $\partial b$ is a multiple of $U^{n+1}$, we see that $\partial(\mathminus(f_{2n+1,-1})(b))$ is also a multiple of $U^{n+1}$. Therefore, we can write 
\[
\partial(\mathminus(f_{2n+1,-1})(b)) = U^{n+1} c.
\]
Recall that $C_n$ is a horizontal almost $\iota_K$-chain complex generated by $a_n,b_n,c_n,d_n,x_n$. Assume that $n>2$, and define a chain map $g\colon C_n \rightarrow CFK^-(S^3,K_{2n+1,-1})$ as
\[
g(a_{n-1})=\mathminus(f_{2n+1,-1})(b),\,g(b_{n-1})= Uc,\,g(c_{n-1})=g(d_{n-1})=0,\, g(x_{n-1})=\mathminus(f_{2n+1,-1})(a).
\]
Then $\iota g + g\iota=0$ as a map from $\widehat{C}_n$ to $\mathminus(K_{2n+1,-1})$, and $g$ is obviously local. Thus $g$ is an almost $\iota_K$-local map, which implies $[C_n] \le \mathfrak{A}(K_{2n+1,-1})$.

Now we recursively construct a sequence $n_1<n_2<\cdots$ of integers as follows. We start with the initial value $n_1=2$. Assuming that we have already determined $n_1,\cdots,n_i$, we take $N_i$ to be an integer such that $N_i \ge n_i$ and $U^{N_i}$ annihilates the torsion submodule $HFK_{tor}(S^3,K_{2n_i+1,-1})$ of $HFK^-(S^3,K_{2n_i+1,-1})$. By \Cref{lem:torsionsplitting} and \Cref{lem:upperbound}, we have $C\mathfrak{A}(K_{2n_i+1,-1}) < [C_{N_i+1}]$ for any integer $C$; this is because $HFK_{tor}(S^3,CK_{2n_i+1,-1})$ is also annihilated by $U^{N_i}$. Then we take $n_{i+1}=N_i+1$ and continue the recursion.

Suppose that the family $\{K_{2n_i+1,-1}\}_{i>0}$ is not linearly independent in $\mathcal{C}$. Then there exist integers $0<i_1<\cdots<i_k$ and $C_1,\ldots,C_k$ such that $C_k > 0$ and
\[
C_k \mathfrak{A}(K_{2n_k+1,-1}) = \sum_{j=1}^{k-1} C_j \mathfrak{A}(K_{2n_{i_j}+1,-1}).
\]
However, since $[C_{n_i}] \le \mathfrak{A}(K_{2n_i+1,-1})$ and $C\mathfrak{A}(K_{2n_i+1,-1}) < [C_{n_{i+1}}]$ for any integer $C$, we should have 
\[
\begin{split}
    \sum_{j=1}^{k-1} C_j \mathfrak{A}(K_{2n_{i_j}+1,-1}) &= C_1 \mathfrak{A}(K_{2n_{i_1}+1,-1}) + \sum_{j=2}^{k-1} C_j \mathfrak{A}(K_{2n_{i_j}+1,-1}) \\
    &< (C_2 +1) \mathfrak{A}(K_{2n_{i_2}+1,-1}) + \sum_{j=3}^{k-1} C_j \mathfrak{A}(K_{2n_{i_j}+1,-1}) \\
    &< (C_{k-1}+1) \mathfrak{A}(K_{2n_{i_{k-1}}+1,-1}) \\
    &< [C_{n_{i_{k-1}+1}}] \le [C_{n_{i_k}}] \le \mathfrak{A}(K_{2n_{i_k}+1,-1}) \le C_k \mathfrak{A}(K_{2n_{i_k}+1,-1}),
\end{split}
\]
a contradiction. Therefore the given family is linearly independent.
\end{proof}

\section{Generalization to iterated cables}
\label{sec:iteratedcable}

Recall from \Cref{lem:almostbordered} that a type-$D$ morphism $f\colon\widehat{CFD}(S^3 \smallsetminus K_1)\rightarrow \widehat{CFD}(S^3 \smallsetminus K_2)$ is \emph{almost bordered-$\iota$-local} if and only if $\mathminus(f)$ is an almost $\iota_K$-local map. The arguments used in the previous section to prove that $K_{2n+1,-1}$ has infinite order in $\mathcal{C}$ can be rephrased as follows. In particular, we do not require $K$ is torsion in $\mathcal{C}$; having a type-$D$ morphism $f$ satisfying the given conditions is enough.

\begin{prop}
\label{prop:proofrephrase}
Let $K$ be a knot. If $\widehat{CFD}(S^3 \smallsetminus K)$ admits a direct summand isomorphic to $\widehat{CFD}(S^3 \smallsetminus O)$ and there exists an almost bordered-$\iota$-local map from $\widehat{CFD}(S^3 \smallsetminus E)$ to $\widehat{CFD}(S^3 \smallsetminus K)$, then for any integer $n>0$, we have $\mathfrak{A}(K_{2n+1,-1}) \ge C_O$ and $\mathfrak{A}(K_{2n+1,-1}) \ge C_E$ in $\mathfrak{I}^U_K$, from which it follows that $\mathfrak{A}(K_{2n+1,-1}) > C_O$, and that $K_{2n+1,-1}$ has infinite order in $\mathcal{C}$.
\end{prop} 


Recall that $CFK_{UV}(S^3,E_{2n+1,-1})$ admits a direct summand $C$, generated by elements $a,b,c,d,e,f,g$, where the differential acts as
\[
\partial b= U^n c+UVd+V^n e,\,\partial c=Vf,\,\partial e=Ug,\,\partial a=\partial d=\partial f=\partial g=0.
\]
Furthermore, we have $\iota_{E_{2n+1,-1}}(b)=b+a$ or $b+a+U^{n-1}f$; so we may assume that $\iota_{E_{2n+1,-1}}(b)=b+a$ possibly after a basis change given by $a\mapsto a+a+U^{n-1}f$. Then we can adjust $\iota_{E_{2n+1,-1}}$ by a homotopy to get
\[
\iota_{E_{2n+1,-1}}(a)=\iota_{E_{2n+1,-1}}(b+\iota_{E_{2n+1,-1}}(b))=b+a+\iota^2_{E_{2n+1,-1}}(b) = a+\Phi\Psi(b)=a+U^{n-1}f.
\]

We claim that the summand $C$ can be made $\iota_{E_{2n+1,-1}}$-stable.
\begin{lem}
\label{lem:cablecomputation}
There exists a chain map, homotopic to $\iota_{E_{2n+1,-1}}$, such that its action on the basis elements $a,b,c,d,e,f,g$ is described as
\[
a\mapsto a+U^{n-1}f,\,b\mapsto b+a,\,c\mapsto e,\,e\mapsto c,\,f\mapsto g,\,g\mapsto f,\,d\mapsto d.
\]
\end{lem}
\begin{proof}
We have already achieved $a\mapsto a+U^{n-1}f$ and $b\mapsto b+a$, so it remains to achieve the rest. Since $\iota_{E_{2n+1,-1}}$ is a chain map, we have
\[
U^n c+V^n e=\partial b=\partial\iota_{E_{2n+1,-1}}(b)= \iota_{E_{2n+1,-1}}\partial(b)=V^n \iota_{E_{2n+1,-1}}(c)+U^n \iota_{E_{2n+1,-1}}(e).
\]
Hence $c+\iota_{E_{2n+1,-1}}(e)$ is a multiple of $V$ and $e+\iota_{E_{2n+1,-1}}(c)$ is a multiple of $U$. However, by considering the bidegrees of $c$ and $C_E$, it is easy to see that the only possibility is $c+\iota_{E_{2n+1,-1}}(e)=e+\iota_{E_{2n+1,-1}}(c)=0$. Therefore we get $\iota_{E_{2n+1,-1}}(c)=e$ and $\iota_{E_{2n+1,-1}}(e)=c$.

We then proceed to remaining basis elements. Since $\partial c=Vf$ and $\partial e=Ug$, we have 
\[
\begin{split}
    U\iota_{E_{2n+1,-1}}(f) & =\iota_{E_{2n+1,-1}}\partial c=\partial \iota_{E_{2n+1,-1}}(c)=\partial e=Ug, \\
    V\iota_{E_{2n+1,-1}}(g) & =\iota_{E_{2n+1,-1}}\partial e=\partial \iota_{E_{2n+1,-1}}(e)=\partial c=Vf,
\end{split}
\]
which implies $\iota_{E_{2n+1,-1}}(f)=g$ and $\iota_{E_{2n+1,-1}}(g)=f$. Furthermore, since $\partial d=U^{n-1}f + V^{n-1}g$, we have 
\[
\partial (d+\iota_{E_{2n+1,-1}}(d))= \partial d+\partial \iota_{E_{2n+1,-1}}(d)=(1+\iota_{E_{2n+1,-1}})(U^{n-1}f+V^{n-1}g)=0,
\]
so $\iota_{E_{2n+1,-1}}(d)$ differs from $d$ by a cycle. Since $d$ has bidegree $(1,1)$ and every other basis elements of the complex $CFK_{UV}(S^3,E_{2n+1,-1})$ lies in a bidegree in which at least one of the two degrees is nonpositive, we should have
\[
\iota_{E_{2n+1,-1}}(d)=d+ \sum (\alpha_i d_{0,i} + \beta_i d_{1,i})
\]
for some $\alpha_i,\beta_i\in \{0,1\}$. However, it is clear that $\sum (\alpha_i d_{0,i} + \beta_i d_{1,i})$ is cycle if and only if all $\alpha_i$ and $\beta_i$ vanish. Therefore we get $\sum (\alpha_i d_{0,i} + \beta_i d_{1,i})(d)=d$.
\end{proof}

Using \Cref{lem:cablecomputation}, we can construct a bordered-$\iota$-local map from the figure-eight knot complement to $S^3 \smallsetminus E_{2n+1,-1}$ as follows.

\begin{lem}
\label{lem:cable-fig8}
For any postive integer $n$, there exists an almost bordered-$\iota$-local map $f$ from $\widehat{CFD}(S^3 \smallsetminus E)$ to $\widehat{CFD}(S^3 \smallsetminus E_{2n+1,-1})$ such that the image of $\mathhat(f)$ is contained in the subcomplex generated by $a,b,c,d,e,f,g$.
\end{lem}
\begin{proof}
It follows from the proof of \cite[Proposition 11.38]{lipshitz2018bordered} that for any two knots $K_1,K_2$ and any chain map $g:CFK_{UV}(S^3,K_1)\rightarrow CFK_{UV}(S^3,K_2)$, there exists a type-D morphism $G:\widehat{CF}(S^3 \smallsetminus K_1)\rightarrow \widehat{CFD}(S^3 \smallsetminus K_2)$ such that $\mathhat(G)\sim \widehat{g}$. Hence the existence of the desired almost bordered-$\iota$-local map $F$ can be shown by proving that there exists an $\iota_K$-local map $f_n$ from $CFK_{UV}(S^3,E)$ to $CFK_{UV}(S^3,E_{2n+1,-1})$, whose image is contained in the subcomplex generated by $a,b,c,d,e,f,g$.

Recall that the knot Floer chain complex for the figure-eight knot $E$ is generated by five elements, $x,h,s,t,z$, where the differential acts by 
\[
\partial h = Us+Vt,\,\partial s = Vz,\,\partial t=Uz,\,\partial z=\partial x=0,
\]
and the involution acts by 
\[
\iota_K(h)=h+x,\,\iota_K(x)=x+z,\,\iota_K(s)=t,\,\iota_K(t)=s,\,\iota_K(z)=z.
\]
For any $n>0$, consider the map $f_n\colon CFK_{UV}(S^3,E)\rightarrow CFK_{UV}(S^3,E_{2n+1,-1})$ defined as
\[
f_n(x)=a,\,f_n(h)=b,\,f_n(s)=U^{n}c+Vd,\,f_n(t)=V^{n}e,\,f_n(z)=V^{n}d.
\]
It is straightforward to check that $\partial f_n + f_n \partial=0$, so that $f_n$ is a local map. Furthermore, by \Cref{lem:cablecomputation}, the map $\iota_{E_{2n+1,-1}} f_n+f_n\iota_{E}$ is given by 
\[
h\mapsto 0,\,s\mapsto Ud,\,t\mapsto Vd,\,z,x\mapsto U^{n}f+V^{n}g.
\]
Consider the linear map $H\colon CFK_{UV}(S^3,E)\rightarrow CFK_{UV}(S^3,E_{2n+1,-1})$ defined by 
\[
H(h)=H(s)=H(t)=0,\,H(z)=H(x)=d.
\]
Then we get $\iota_{E_{2n+1,-1}} f_n+f_n\iota_{E} = \partial H+H\partial$, so we deduce that $f_n$ is an $\iota_K$-local map. Furthermore, the image of $f_n$ consists of $\mathbb{F}_2[U]$-linear combinations of $a,b,c,d,e,f,g$. Therefore $f_n$ satisfies the desired properties.
\end{proof}

Using the almost bordered-$\iota$-local map we got from \Cref{lem:cable-fig8}, one can now prove the following lemma, which will be used as an inductive step in the proof of \Cref{thm:mainthm}.

\begin{lem}
\label{lem:bootstrap}
Let $K$ be a knot and $n$ be a positive integer. If there exists an almost bordered-$\iota$-local map $f\colon\widehat{CFD}(S^3 \smallsetminus E)\rightarrow \widehat{CFD}(S^3 \smallsetminus K)$, then there also exists an almost bordered-$\iota$-local map from $\widehat{CFD}(S^3 \smallsetminus E)$ to $\widehat{CFD}(S^3 \smallsetminus K_{2n+1,-1})$.
\end{lem}
\begin{proof}
Recall from \Cref{sec:cable} that the type-$D$ morphism
\[
f_{2n+1,-1}\colon\widehat{CFD}(S^3 \smallsetminus E_{2n+1,-1})\rightarrow \widehat{CFD}(S^3 \smallsetminus K_{2n+1,-1})
\]
satisfies $\iota_{K_{2n+1,-1}}(\mathhat(f_{2n+1,-1})(b))=\mathhat(f_{2n+1,-1})(b+a)$ and $\iota_{K_{2n+1,-1}}(\mathhat(f_{2n+1,-1})(a))=\mathhat(f_{2n+1,-1})(a)$; note that we do not know whether $f_{2n+1,-1}$ is an almost bordered-$\iota$-local map. Consider the bordered-$\iota$-local map
\[
F\colon\widehat{CFD}(S^3 \smallsetminus E)\rightarrow \widehat{CFD}(S^3 \smallsetminus E_{2n+1,-1})
\]
given in \Cref{lem:cable-fig8}. Using the notation from the proof of \Cref{lem:cable-fig8} that $CFK_{UV}(S^3,E)$ is generated by $x,h,s,t,z$, since the image of $\mathhat(F)$ is contained in the $\mathbb{F}_2$-linear span of $a,b,c,d,e,f,g$, it is clear that $\mathhat(F)$ maps $x$ and $h$ to $a$ and $b$, respectively, and maps $s,t,z$ to 0. Hence the map 
\[
\mathhat(f_{2n+1,-1}\circ F) = \mathhat(f_{2n+1,-1})\circ \mathhat(F)
\]
satisfies $\iota_{K_{2n+1,-1}}\circ \mathhat(f_{2n+1,-1}\circ F) \sim \mathhat(f_{2n+1,-1}\circ F)\circ \iota_{E}$, and therefore $f_{2n+1,-1} \circ F$ is almost bordered-$\iota$-local.
\end{proof}

\begin{lem}
\label{lem:cabletrivial}
Let $K$ be a knot such that $\widehat{CFD}(S^3 \smallsetminus K)$ admits a direct summand isomorphic to $\widehat{CFD}(S^3 \smallsetminus O)$. Then for any pattern $P\subset T_\infty$ such that $P(O)$ is slice, $CFK_{\mathcal{R}}(S^3,P(K))$ also admits a direct summand isomorphic to $\widehat{CFD}(S^3 \smallsetminus O)$.
\end{lem}
\begin{proof}
By \cite[Theorem 11.26]{lipshitz2018bordered}, we know that $\widehat{CFD}(S^3,K)$ admits a summand isomorphic to $\widehat{CFD}(S^3,O)$. Then, by taking a box tensor product with the bimodule $\widehat{CFDA}(T_\infty \smallsetminus P)$, we see that $\widehat{CFD}(S^3\smallsetminus P(K))$ admits a summand isomorphic to $\widehat{CFD}(S^3\smallsetminus P(O))$. Also, since $P(O)$ is slice, it follows from \cite[Theorem 1]{hom2017survey} that $CFK_{UV}(S^3,P(O))$ has a direct summand isomorphic to $\mathbb{F}_2[U,V]$ concentrated in bidegree $(0,0)$, from which it follows from \cite[Theorem 11.26]{lipshitz2018bordered} that $\widehat{CFD}(S^3 \smallsetminus P(O))$ admits a direct summand isomorphic to $\widehat{CFD}(S^3 \smallsetminus O)$.
\end{proof}

We are now ready to prove \Cref{thm:mainthm}, whose statement we recall.

\begin{thm}
\label{thm:mainthmbody}
Denote by $\mathcal{C}_{T}$ the torsion subgroup of the knot concordance group $\mathcal{C}$. Then there is a nontrivial group homomorphism 
$\mathfrak{A}\colon \mathcal{C}_{T}\rightarrow \mathbb{Z}/2\mathbb{Z}$.
Moreover, if $K$ is torsion in $\mathcal{C}$ with $\mathfrak{A}(K) = 1$, then for any sequence of positive integers $n_1, n_2, \ldots, n_m$ the iterated cable $K_{2n_1+1,1;2n_2+1,1;\ldots ;2n_m+1,1}$ has infinite order in $\mathcal{C}$.
\end{thm}

\begin{proof}
The existence of the nontrivial group homomorphism follows from \Cref{cor:pseudoArf}. Also, as in the proof of \Cref{thm:odd1hasinfiniteorder}, it is enough to prove that $K_{2n_1+1,-1;2n_2+1,-1;\ldots ;2n_m+1,-1}$ has infinite order in $\mathcal{C}$.

In order to do so, we show that there exists an almost bordered-$\iota$-local map $$\widehat{CFD}(S^3 \smallsetminus E) \rightarrow \widehat{CFD}(S^3 \smallsetminus K_{2n_1+1,-1;2n_2+1,-1; \ldots, 2n_{m-1}+1,-1}),$$ from which the theorem follows by \Cref{lem:cabletrivial} and \Cref{prop:proofrephrase}. We already know, possibly after replacing $K$ by another knot concordant to it, that the result holds when $m=1$, as shown in \Cref{lem:borderedinterpret}. Moreover, \Cref{lem:bootstrap} implies that the claim holds for any positive integer $m$.\end{proof}

\begin{rem}
\label{rem:iterrmk}
In particular, the proof of \Cref{thm:mainthmbody} implies that $\mathfrak{A}(K_{2n_1+1,-1;2n_2+1,-1; \ldots, 2n_{m-1}+1,-1})>0$. 
\end{rem}


\section{Strongly rationally slice knots that are not slice}
\label{sec:stronglyrationallyslice}

In this section, we prove \Cref{thm:mainthm-complexity} using \Cref{thm:mainthmbody} and the following two lemmas. For simplicity, we first define the \emph{complexity} of rationally slice knots as follows. 

\begin{defn}\label{def:complexity}
Let $K$ be a rationally slice knot that bounds a smoothly embedded disk $\Delta$ in a rational homology ball $X$. The inclusion induces a homomorphism
$$H_1(S^3 \smallsetminus K;\mathbb{Z}) \cong \mathbb{Z} \to H_1(X \smallsetminus \Delta; \mathbb{Z})/\text{torsion} \cong \mathbb{Z},$$
which is given by multiplication by a nonzero integer $c$, and we may assume further that $c$ is positive by changing orientations if necessary. We define the \emph{complexity} of $\Delta$ to be the positive integer $c$ and the \emph{complexity} of $K$ to be the minimum complexity of all possible rational slice disks for $K$.\end{defn}

Note that a rationally slice knot has complexity one if and only if it is strongly rationally slice. The following lemma allows us to construct many strongly rationally slice knots. We use $\nu Y$ to denote a open tubular neighborhood of a submanifold $Y$.

\begin{lem}
\label{lem:satellitecomplexity}
Let $K$ be a complexity $c$ rationally slice knot and $P\subset T_\infty$ be a winding number $w$ pattern. If $P(O)$ is slice, then $P(K)$ is rationally slice knot of complexity at most $c/\gcd(c,w)$.
\end{lem}
\begin{proof}
Through out the proof all homology groups have integral coefficients. Let $\Delta$ be a complexity $c$ smoothly embedded disk in a rational homology ball $X$. Then the disk $\Delta$ induces a rational concordance $C \subset X \smallsetminus B^4$ from $K$ to the unknot $O$. Using the Mayer-Vietoris sequence, it can be easily verified that the homomorphism induced by the inclusion 
$$
    H_1(S^3 \smallsetminus K) \cong \mathbb{Z} \to H_1\left(X \smallsetminus \left(B^4 \cup C\right)\right)/\text{torsion} \cong \mathbb{Z}
$$
is given by multipication by $c$. Moreover, using the pattern $P$, we get a rational concordance $C_P \subset X\smallsetminus B^4$ from $P(K)$ to the slice knot $P(O)$. Combined with a slice disk $\Delta_{P(O)}$ for $P(O)$, we get a rational slice disk
$$\Delta_{P(K)}:= C_P \cup \Delta_{P(O)} \subset \left(X \smallsetminus B^4\right)  \cup B^4= X$$
for $P(K)$. Then as before we also have that  
$$
    H_1(S^3 \smallsetminus P(K)) \cong \mathbb{Z} \to H_1\left(X \smallsetminus \left(B^4 \cup C_P\right)\right)/\text{torsion} \cong \mathbb{Z}
$$
is given by multiplication by the complexity of $\Delta_{P(K)}$. Hence, it is enough to analize this homomorphism induced by the inclusion (see $v_3$ below).


Note that from the following Mayer-Vietoris sequence
\[
0 \to H_1(P \times S^1) \rightarrow H_1(T_{\infty} \smallsetminus \nu P) \oplus H_1(\nu P) \to H_1(T_\infty) \to 0,
\]
it can be easily verified that $H_1(T_{\infty} \smallsetminus \nu P) \cong \mathbb{Z}\oplus \mathbb{Z} \cong \langle \lambda_T, \mu_P\rangle,$
where $$\lambda_T := S^1 \times \{\pt \} \subset \partial T_\infty  \subset T_\infty = S^1 \times D^2 \qquad\text{and}\qquad \mu_P:=\{\pt \} \times \nu P = P \times D^2 \subset  T_\infty.$$ 
Now, we consider the following commutative diagram of Mayer-Vietoris sequences,
\[
\xymatrixcolsep{2pc}\xymatrix{
0\ar[r] & 
H_1(K \times S^1)\ar[r]^-{h_1}\ar[d]^{v_1} &
H_1(T_\infty \smallsetminus P) \oplus H_1(S^3 \smallsetminus \nu K)\ar[r]^-{h_2}\ar[d]^{v_2} &
H_1(S^3 \smallsetminus P(K))\ar[r]\ar[d]^{v_3} &
0\\
0\ar[r] &
H_1(C \times S^1)\ar[r]^-{h_3} & 
H_1(\nu C \smallsetminus C_P) \oplus H_1\left(X \smallsetminus( B^4 \cup \nu C)\right)\ar[r]^-{h_4} & 
H_1\left(X \smallsetminus \left(B^4 \cup C_P \right)\right)\ar[r] 
& 0}
\]
where the vertical homomorphisms are induced by the inclusions. Let $\mu$ be the meridian of $P(K)$. We make the following observations:
\begin{itemize}
    \item $h_1([K \times \{\pt\}]) = ([\lambda_T],0) \in H_1(T_\infty \smallsetminus P) \oplus H_1(S^3 \smallsetminus \nu K)$.
    \item $h_1([\{\pt\} \times S^1]) = (w [\mu_P],1) \in H_1(T_\infty \smallsetminus P) \oplus H_1(S^3 \smallsetminus \nu K)$.
    \item $h_2\left(([\mu_P],0)\right) = [\mu] \in H_1(S^3 \smallsetminus P(K))$.
    \item $v_1$ is an isomorphism.
    \item $v_2{\big|}_{H_1(T_\infty \smallsetminus P)} \colon H_1(T_\infty \smallsetminus P) \to H_1(\nu C \smallsetminus C_P)$ is an isomorphism.
    \item $v_2{\big|}_{H_1(S^3 \smallsetminus \nu K)} \colon H_1(T_\infty \smallsetminus P) \to H_1\left(X \smallsetminus( B^4 \cup \nu C)\right)/\text{torsion}$ is given by multiplication by $c$.\end{itemize}
Hence, we see that $$\image(h_3)\cong \langle ([\lambda_T],0) , (w[\lambda_P],c))\rangle \subset H_1(\nu C \smallsetminus C_P) \oplus H_1\left(X \smallsetminus( B^4 \cup \nu C)\right)/\text{torsion} \cong \mathbb{Z} \oplus \mathbb{Z} \oplus \mathbb{Z}.$$
This implies that $$H_1\left(X \smallsetminus( B^4 \cup C_P)\right)/\text{torsion}\cong \mathbb{Z} \oplus \mathbb{Z}/\langle(w,c)\rangle,$$ and $v_3([\mu]) = h_4 \circ v_2 (([\mu_P],0))$ is represented by $(1,0)\in \mathbb{Z} \oplus \mathbb{Z}/\langle(w,c)\rangle,$ which completes the proof.\end{proof}

\begin{lem}
\label{lem:surgerylem}
Let $K$ be a knot. If the connected sum of some copies of $S^3 _{+1}(K)$ bounds a spin rational homology ball, then there exists an almost $\iota_K$-local map from $\mathbb{F}_2[U,V]$ to $CFK_{UV}(S^3,K)$, and thus $\mathfrak{A}(K) \ge [C_O]$.
\end{lem}
\begin{proof}
The assumption, together with the fact that the almost $\iota$-local equivalence group is totally ordered, implies that the $CF^- (S^3 _{+1}(K))$ is almost $\iota$-locally trivial; we refer to \cite{dai2018infinite} for the definitions of almost $\iota$-complexes and almost $\iota$-local equivalence classes and \cite[Theorem 3.25]{dai2018infinite} for the fact that the almost $\iota$-local equivalence group is totally ordered. We know from \cite[Theorem 1.6]{hendricks2020surgery} that the almost $\iota$-local equivalence class of $CF^- (S^3 _{+1}(K))$ is given by $(A_0(K),\iota)$, where $A_0(K)\subset CFK_{UV}(S^3,K)$ is the subset of elements of Alexander grading 0 and $\iota$ is the restriction of the action of $\iota_K$ on $CFK_{UV}(S^3,K)$ to $A_0 (K)$. Hence $(A_0(K),\iota)$ is almost $\iota$-locally trivial.

Choose an almost $\iota$-local map $f\colon\mathbb{F}_2[U]\rightarrow A_0(K)$. Recall that the $U$-action on $A_0(K)$ corresponds to the $UV$-action on $CFK_{UV}(S^3,K)$. Thus we can extend $f$ by tensoring with $\mathbb{F}_2[U,V]$, along the ring inclusion $\mathbb{F}_2[UV]\rightarrow \mathbb{F}_2[U,V]$, to get a local map 
\[
\tilde{f}\colon\mathbb{F}_2[U,V]\rightarrow A_0(K)\otimes_{\mathbb{F}_2[UV]} \mathbb{F}_2[U,V] \xrightarrow{\text{multiplication map}} CFK_{UV}(S^3,K).
\]
Since $\iota \circ f \sim f$ mod $U$ and the action of $\iota$ on $A_0(K)$ is the restriction of the action of $\iota_K$ on $CFK_{UV}(S^3,K)$, it is clear that $\iota_K \circ \tilde{f} \sim \tilde{f}$ mod $UV$, so $\tilde{f}$ is an almost $\iota_K$-local map.
\end{proof}

We are finally able to prove \Cref{thm:mainthm-complexity}, whose statement we recall.

\begin{thm}
\label{thm:mainthm-complexitybody}
If $K$ is torsion in $\mathcal{C}$ with $\mathfrak{A}(K) = 1$, then for any sequence of positive integers $n_1, n_2, \ldots, n_m$ the $(2,1)$-cable of the iterated cable $K_{2n_1+1,1;2n_2+1,1; \ldots; 2n_m+1,1}$ has infinite order in $\mathcal{C}$.
\end{thm}


\begin{proof}
Let $J$ be the iterated cable $K_{2n_1+1,1;2n_2+1,1; \ldots; 2n_m+1,1}$ and $\Sigma(J_{2,1})$ be the 2-fold cyclic branched cover of $S^3$ branched along $J_{2,1}$. By applying the Akbulut-Kirby~\cite{Akbulut-Kirby:1978-1} technique to $J_{2,1}$, we have that
\[
\Sigma\left(J_{2,1}\right) \cong S^3 _{+1} \left(J\# J^r\right).
\]

Suppose that $nJ_{2,1}$ is slice for some positive integer $n$ and let $D$ be its slice disk. Then the 2-fold cyclic branched cover of $B^4$ branched along $D$ is a rational homology ball with a spin structure that restricts to the unique spin structure on its boundary $\Sigma\left(nJ_{2,1}\right)\cong n\Sigma\left(J_{2,1}\right)$. By \Cref{lem:surgerylem}, this implies
\[
\mathfrak{A}\left(K_{2n_1+1,1;2n_2+1,1; \ldots; 2n_m+1,1}\# K^r_{2n_1+1,1;2n_2+1,1; \ldots; 2n_m+1,1}\right)\ge [C_O].
\]
However, we have 
\[
\mathfrak{A}\left((-K)_{2n_1+1,-1;2n_2+1,-1; \ldots; 2n_m+1,-1}\right)> [C_O] \qquad\text{and}\qquad \mathfrak{A}\left((-K)^r_{2n_1+1,1;2n_2+1,-1; \ldots; 2n_m+1,-1}\right) > [C_O]
\]
by \Cref{prop:reversal} and \Cref{rem:iterrmk}, which gives
\[
\mathfrak{A}\left(K_{2n_1+1,1;2n_2+1,1; \ldots; 2n_m+1,1}\# K^r_{2n_1+1,1;2n_2+1,1; \ldots; 2n_m+1,1}\right)< [C_O];
\]
a contradiction. Hence $J_{2,1}$ has infinite order in $\mathcal{C}$.
\end{proof}

\bibliographystyle{amsalpha}
\bibliography{ref}
\end{document}